\documentclass[11pt]{amsart}

\usepackage{cmlgc}
\usepackage{ucs}
\usepackage[english]{babel}
\usepackage{bm,amsfonts,amsmath,amssymb,dsfont} 
\usepackage{graphicx}
\usepackage[utf8x]{inputenc}
\usepackage[T1]{fontenc}
\usepackage{enumerate}

\usepackage[babel=true]{csquotes}

\usepackage{pdfsync}

\makeatletter
\def\section{\@startsection{section}{1}\z@{.9\linespacing\@plus\linespacing}%
  {.7\linespacing} {\fontsize{13}{15}\selectfont\scshape\centering}}
\def\paragraph{\@startsection{paragraph}{4}%
  \z@{0.3em}{-.5em}%
  {$\bullet$ \ \normalfont\itshape}}
\makeatother

\newtheorem{theo}{Theorem}[section]
\newtheorem{prop}[theo]{Proposition}
\newtheorem{lem}[theo]{Lemma}
\newtheorem{cor}[theo]{Corollary}

\theoremstyle{definition}

\newtheorem{notation}[theo]{Notation}

\renewcommand{\geq}{\geqslant}
\renewcommand{\leq}{\leqslant}

\theoremstyle{remark}
\newtheorem{rem}[theo]{Remark}

\makeatletter

\@addtoreset{equation}{section}
\makeatother

\usepackage{float}
\usepackage{color}

\definecolor{gr}{rgb}   {0.,   0.69,   0.23 }
\definecolor{bl}{rgb}   {0.,   0.5,   1. }
\definecolor{mg}{rgb}   {0.85,  0.,    0.85}
\definecolor{yl}{rgb}   {0.8,  0.7,   0.}

\definecolor{webred}{rgb}{0.75,0,0}
\definecolor{webgreen}{rgb}{0,0.75,0}
\usepackage[citecolor=webgreen,colorlinks=true,linkcolor=webred]{hyperref}

\newcommand{\la}{\lambda}
\newcommand{\eps}{\varepsilon}

\newcommand{\A}{\mathbf{A}}

\newcommand{\N}{\mathbb{N}}
\newcommand{\Z}{\mathbb{Z}}
\newcommand{\R}{\mathbb{R}}
\newcommand{\C}{\mathbb{C}}

\newcommand{\Id}{\mathsf{Id}}

\newcommand{\Op}{\mathsf{Op}}

\newcommand{\ad}{\mathsf{ad}}

\newcommand{\T}{\mathsf{T}}

\newcommand{\deriv}[2]{\frac{\partial #1}{\partial #2}}
\newcommand{\norm}[1]{\left\|#1\right\|}
\newcommand{\Cinf}{C^\infty}
\newcommand{\pscal}[2]{\langle #1,#2\rangle}
\newcommand{\trsp}{\raisebox{.6ex}{${\scriptstyle t}$}}
\newcommand{\restr}{\upharpoonright}
\newcommand{\abs}[1]{\left|#1\right|}
\newcommand{\phy}{\varphi}
\newcommand{\h}{\hbar}
\newcommand{\ham}[1]{\mathcal{X}_{#1}}

\def\dx{\, {\rm d}}

\title[Geometry and Spectrum in 2D Magnetic Wells]{\large Geometry and
  Spectrum in 2D Magnetic Wells} \author{Nicolas Raymond and San V\~u
  Ng\d{o}c} \address{~\newline Nicolas Raymond: IRMAR (UMR 6625), Universit{\'e}
  de Rennes 1,
  Campus de Beaulieu, 35042 Rennes cedex (France)\newline
  San V\~u Ng\d{o}c: IRMAR (UMR 6625), Universit{\'e} de Rennes 1,
  Campus de Beaulieu, 35042 Rennes cedex (France) \& Institut
  Universitaire de France.}

\date{\today}

\begin{document}
\begin{abstract}
  This paper is devoted to the classical mechanics and spectral
  analysis of a pure magnetic Hamiltonian in $\R^2$. It is established
  that both the dynamics and the semiclassical spectral theory can be
  treated through a Birkhoff normal form and reduced to the study of a
  family of one dimensional Hamiltonians.  As a corollary, recent
  results by Helffer-Kordyukov are extended to higher energies.
\end{abstract}
\maketitle

\section{Introduction}

We consider in this article a charged particle in $\R^2$ moving under
the action of a non-vanishing, time-independent magnetic field which
is orthogonal to the plane. We will study both the classical and
quantum (non relativistic) cases, in a regime where the energy is low
but the magnetic field is strong.

This problem has given rise to many semiclassical investigations in
the last fifteen years. Most of them are motivated by the study of the
Ginzburg-Landau functional and its third critical field $H_{C_{3}}$
which can be related to the lowest eigenvalue of the magnetic
Laplacian (see \cite{FouHel10}). Many cases involving the geometry of
the possible boundary and the variations of the magnetic field have
been analyzed (see \cite{Montgomery95, HelMo96, HelMo01,HelMo04,
  HelKo09, Ray09, FouPer11, HelKo11, HelKo13}). Due to the initial motivation,
most of the papers provide only asymptotic expansions of the lowest
eigenvalue and do not provide the corresponding approximation for the
eigenfunctions. The only paper which explicitly tackles the
approximation of the eigenfunctions and their microlocal properties is
\cite{FouHel06a}, where the authors combine pseudo-differential
techniques and a Grushin reduction. More recently, the contributions
\cite{Ray11b, DomRay13, PoRay13} display that the magnetic $2$-form
and the geometry combine in the semiclassical limit to produce very
fine microlocalization properties for the eigenfunctions. In
particular, it is shown, in various geometric and magnetic settings,
that a normal form procedure can reveal a double scale structure of
the magnetic Laplacian, which is reminiscent of the famous
Born-Oppenheimer approximation. It also established that an effective
electric operator generates asymptotic series for the lowest
eigenpairs. Such results suggest the fact that a full Birkhoff normal
form analysis in the spirit of \cite{Vu06, VuCha08, Vu09} could be
implemented for the magnetic Laplacian. 

This is a remarkable fact that the Birkhoff procedure has never been
implemented to enlighten the effect of magnetic fields on the low
lying eigenvalues of the magnetic Laplacian. A reason might be that,
compared to the case of a Schrödinger operator with an electric
potential, the magnetic case presents a major difficulty: the symbol
itself is not enough to confine the dynamics in a compact
set. Therefore, it is not possible to start with a simple harmonic
approximation at the principal level. This difficulty can be seen in
the recent papers by Helffer and Kordyukov \cite{HelKo11} (dimension
two) and \cite{HelKo13} (dimension three) which treat cases without
boundary. In dimension three they provide accurate constructions of
quasimodes, but do not establish the asymptotic expansions of the
eigenvalues which is still an open problem. In dimension two, they
prove that if the magnetic field has a unique and non-degenerate
minimum, the $j$-th eigenvalue admits an expansion in powers of
$\h^{1/2}$ of the form:
\[
\lambda_j(\h) \sim \h \min_{q\in\R^2} B(q) + \h^2(c_1(2j-1)+c_0) + O(\h^{5/2}),
\]
where $c_{0}$ and $c_{1}$ are constants depending on the magnetic
field. In this paper, we extend their result by obtaining a complete
asymptotic expansion --- without odd powers of $\h^{1/2}$ (see
Corollary~\ref{cor:low})--- which actually applies to more general
magnetic wells --- see for instance Corollary~\ref{coro:bs}.

\vskip 1em

Let us describe now the methods and results of the paper. As we shall
recall below, a particle in a magnetic field has a fast rotating
motion, coupled to a slow drift. It is of course expected that the
long-time behaviour of the particle is governed by this drift. We show
in this article that it is indeed the case, and that the drift motion
can be obtained by a one degree of freedom Hamiltonian system, both in
the classical or the quantum setting. What's more, the effective
Hamiltonian is, for small energies, approximated by the magnetic field
itself.

In order to achieve this, we obtain a normal form that explicitly
reduces the study of the original system to a one degree of freedom
Hamiltonian. In the classical case, this gives an approximation of the
dynamics for long times, of order $\mathcal{O}(1/E^\infty)$, where $E$
is the energy. In the quantum case, this gives a complete asymptotic
expansion of the eigenvalues up to $\mathcal{O}(\hbar^\infty)$, where
$\hbar$ is the semiclassical parameter (Planck's constant).

\subsection*{Classical dynamics}
 
Let $(e_1,e_2,e_3)$ be an orthonormal basis of $\R^3$. Our
configuration space is $\R^2=\{q_1 e_1 + q_2 e_2;\;
(q_1,q_2)\in\R^2\}$, and the magnetic field is
$\vec{B}=B(q_1,q_2)e_3$. For the moment we only assume that
$q:=(q_1,q_2)$ belongs to an open set $\Omega$ where $B$ does not
vanish.

With appropriate constants, Newton's equation for the particle under
the action of the Lorentz force writes
\begin{equation}
  \ddot{q} = 2 \dot{q} \wedge \vec B.
  \label{equ:newton}
\end{equation}
The kinetic energy $E=\frac{1}{4}\norm{\dot q}^2$ is conserved. If the
speed $\dot{q}$ is small, we may linearize the system, which amounts
to have a constant magnetic field. Then, as is well known, the
integration of Newton's equations gives a circular motion of angular
velocity $\dot{\theta}=-2B$ and radius $\norm{\dot q}/2B$. Thus, even
if the norm of the speed is small, the angular velocity may be very
important. Now, if $B$ is in fact not constant, then after a while,
the particle may leave the region where the linearization is
meaningful. This suggests a separation of scales, where the fast
circular motion is superposed with a slow motion of the center
(Figure~\ref{fig:beam}).
\begin{figure}[h] 
  \label{fig:beam}
  \centering
  \includegraphics[width=0.5\textwidth]{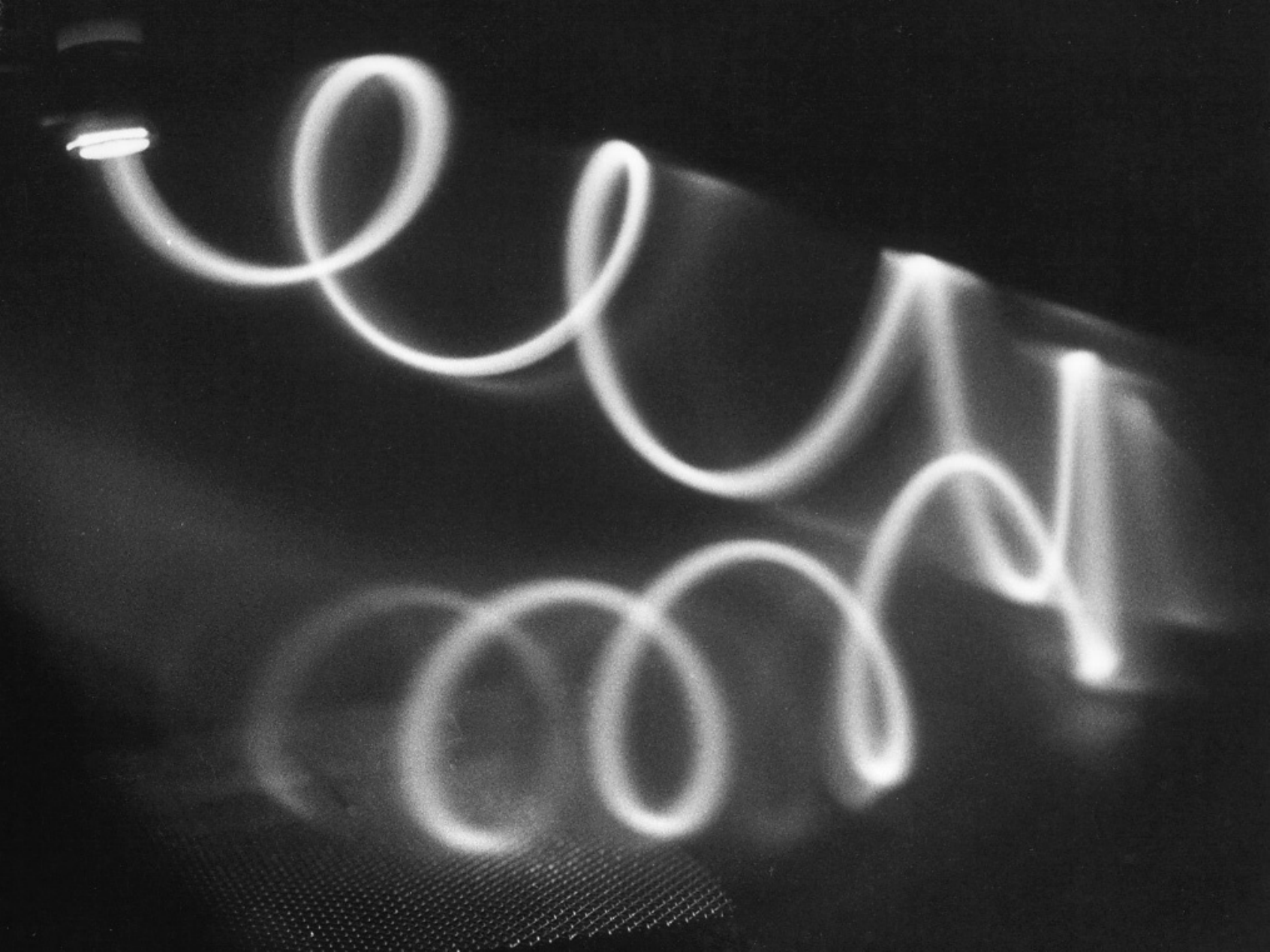}
  \caption[electron beam in a non-uniform magnetic field]{\small This
    photograph shows the motion of an electron beam in a non-uniform
    magnetic field. One can clearly see the fast rotation coupled with
    a drift.  In the magnetic literature, the turning point (here on
    the right), due to the projection of the phase space motion onto
    the position space, is called a \emph{mirror
      point}. {\footnotesize Credits: Prof. Reiner Stenzel,}
    {\footnotesize
      \url{http://www.physics.ucla.edu/plasma-exp/beam/BeamLoopyMirror.html}}}
\end{figure}

It is known that the system~\eqref{equ:newton} is Hamiltonian. In
fact, the Hamiltonian is simply the kinetic energy, but the definition
of the phase space requires the introduction of a magnetic
potential. Let $\mathbf{A}\in\Cinf(\R^2,\R^2)$ such that
\[
\vec B = \nabla \wedge \mathbf{A}.
\]
We may identify $\mathbf{A}=(A_1,A_2)$ with the 1-form
$A=A_1dq_1+A_2dq_2$. Then, as a differential 2-form, $dA =
(\deriv{A_2}{q_1}-\deriv{A_1}{q_2})dq_1\wedge dq_2 = B dq_1\wedge
dq_2$. Thus, by Poincaré lemma we see that, given any smooth magnetic
function $B(q_1,q_2)$, such a potential $\mathbf{A}$ always exists.

In terms of canonical variables $(q,p)\in T^*\R^2=\R^4$ the
Hamiltonian of our system is
\begin{equation}
  H(q,p) = \norm{p-A(q)}^2.
  \label{equ:hamiltonian}
\end{equation}
We use here the Euclidean norm on $\R^2$, which allows the
identification of $\R^2$ with $(\R^2)^*$ by
\begin{equation}
  \forall (v,p)\in\R^2\times (\R^2)^*, \qquad p(v) = \pscal{p}{v}.
  \label{equ:identification}
\end{equation}
Thus, the canonical symplectic structure $\omega$ on $T^*\R^2$ is
given by
\begin{equation}
  \omega((Q_1,P_1), (Q_2,P_2)) = \pscal{P_1}{Q_2} - \pscal{P_2}{Q_1}.
  \label{equ:omega}
\end{equation}

It is easy to check that Hamilton's equations for $H$ imply Newton's
equation~\eqref{equ:newton}. In particular, through the
identification~\eqref{equ:identification} we have $\dot q = 2(p-A)$.

\subsection*{Main results}
We can now state our main results. We consider first large time classical
dynamics. Indeed, while it is quite easy to find an approximation of
the dynamics for finite time, the large time problem has to face the
issue that the conservation of the energy $H$ is not enough to confine
the trajectories in a compact set: the set $H^{-1}(E)$ is not
bounded. 

The first result shows the existence of a smooth symplectic
diffeomorphism that transforms the initial Hamiltonian into a normal
form, up to any order in the distance to the zero energy surface.
\begin{theo}
  \label{theo:classical}
  Let
  \[
  H(q,p):=\norm{p-\mathbf{A}(q)}^2, \quad (q,p)\in T^* \R^2 =
  \R^2\times \R^2,
  \]
  where the magnetic potential $\mathbf{A}:\R^2\to\R^2$ is smooth. Let
  $B:=\deriv{A_2}{q_1} - \deriv{A_1}{q_2}$ be the corresponding
  magnetic field.  Let $\Omega\subset\R^2$ be an open set where $B$
  does not vanish. Then there exists a symplectic diffeomorphism
  $\Phi$, defined in an open set
  $\tilde{\Omega}\subset\C_{z_1}\times\R^2_{z_2}$, with values in $T^*\R^2$,
  which sends the plane $\{z_1=0\}$ to the surface $\{H=0\}$, and such
  that
  \begin{equation}
    H\circ \Phi = \abs{z_1}^2 f(z_2,\abs{z_1}^2) +
    \mathcal{O}(\abs{z_1}^\infty),
    \label{equ:forme-normale}  
  \end{equation}
  where $f:\R^2\times\R \to \R$ is smooth.  Moreover, the map
  \begin{equation}
    \phy: \Omega \ni q \mapsto  \Phi^{-1}(q,\mathbf{A}(q)) \in
    (\{0\} \times \R^2_{z_2} ) \cap \tilde\Omega
    \label{equ:phy}
  \end{equation}
  is a local diffeomorphism and
  \[
  f\circ(\phy(q),0) = \abs{B(q)}.
  \]
\end{theo}
In the following theorem we denote by $K=\abs{z_1}^2
f(z_2,\abs{z_1}^2) \circ\Phi^{-1}$ the (completely integrable) normal
form of $H$ given be Theorem~\ref{theo:classical} above. Let
$\phy_H^t$ be the Hamiltonian flow of $H$, and let $\phy_K^t$ be the
Hamiltonian flow of $K$. Since $K$ has separated variables, it is easy
to compute its flow. The following result ensures that $\phy_K^t$ is a
very good approximation to $\phy_H^t$ for large times.

\begin{theo}
  \label{theo:confining}
  Assume that the magnetic field $B>0$ is confining: there exists
  $C>0$ and $M>0$ such that $B(q)\geq C$ if $\norm{q}\geq M$. Let $C_0
  < C$. Then 
  \begin{enumerate}
  \item The flow $\phy_H^t$ is uniformly bounded for all starting
    points $(q,p)$ such that $B(q)\leq C_0$ and
    $H(q,p)=\mathcal{O}(\epsilon)$ and for times of order
    $\mathcal{O}(1/\epsilon^N)$, where $N$ is arbitrary.
  \item Up to a time of order
    $T_\epsilon=\mathcal{O}(\abs{\ln\epsilon})$, we have
    \begin{equation}
      \norm{\phy_H^t(q,p) - \phy_K^t(q,p)} = \mathcal{O}(\epsilon^\infty)
      \label{equ:two-flows}  
    \end{equation}
    for all starting points $(q,p)$ such that $B(q)\leq C_0$ and
    $H(q,p)=\mathcal{O}(\epsilon)$.
  \end{enumerate}
\end{theo}
It is interesting to notice that, if one restricts to regular values
of $B$, one obtains the same control for a much longer time, as stated
below.
\begin{theo}\label{theo:confining2}
  Under the same confinement hypothesis as
  Theorem~\ref{theo:confining}, let $J\subset(0,C_0)$ be a closed
  interval such that $dB$ does not vanish on $B^{-1}(J)$. Then up to a
  time of order $T=\mathcal{O}(1/\epsilon^N)$, for an arbitrary $N>0$,
  we have
  \[
  \norm{\phy_H^t(q,p) - \phy_K^t(q,p)} = \mathcal{O}(\epsilon^\infty)
  \]
  for all starting points $(q,p)$ such that $B(q)\in J$ and
  $H(q,p)=\mathcal{O}(\epsilon)$.
\end{theo}
It is possible that the longer time $T=\mathcal{O}(1/\epsilon^N)$
reached in~\eqref{equ:two-flows} could apply as well for some types of
singularities of $B$; this seems to be an open question at the moment.

We may now describe the magnetic dynamics in terms of a fast rotating
motion with a slow drift. In order to do this, we introduce the
adiabatic action
\[
I :=\abs{z_1}^2 = \int_\gamma pdq,
\]
where $\gamma$ is the loop corresponding to the fast motion (which we
can obtain by using a local approximation by a constant magnetic
field). Since $\{I,K\}=0$, $I$ is a constant of motion for the flow
$\phy_K^t$. Moreover, the Hamiltonian flow of $I$ generates a
$2\pi$-periodic $S^1$ action on the level set
$\{I=\text{const}\}$. For $I\neq 0$, the reduced symplectic manifold
$\Sigma_I:=\{I=\text{const}\}/S^1$ may be identified with
$\Sigma:=I^{-1}(0)= H^{-1}(0)$, endowed with the symplectic form
$d\xi_2\wedge dx_2$. (As we shall see in Lemma~\ref{lem:Sigma} below,
we may also identify $\Sigma$ with $\R^2_{(q_1,q_2)}$ endowed with the
symplectic form $B dq_1\wedge dq_2$.) Then, for each value of $I$, the
function $K$ defines a Hamiltonian $h_I$ on $\Sigma$:
\[
h_I(z_2) := I f(z_2,I).
\]
In the next statement, we assume that $B$ is confining and we denote
by $T(\epsilon)$ the time given by Theorems~\ref{theo:confining}
or~\ref{theo:confining2}, depending on the initial value of $B$. In
view of the fact that the Hamiltonian vector field of $K$ splits into
the sum of commuting vector fields
\[
\ham{K} = f\ham{I} + I\ham{f(z_2,I)}, 
\]
we immediately obtain the following corollary, which is illustrated by
Figure~\ref{fig:numerics}.

\begin{figure}[h]
  \centering
  \includegraphics[width=0.5\textwidth]{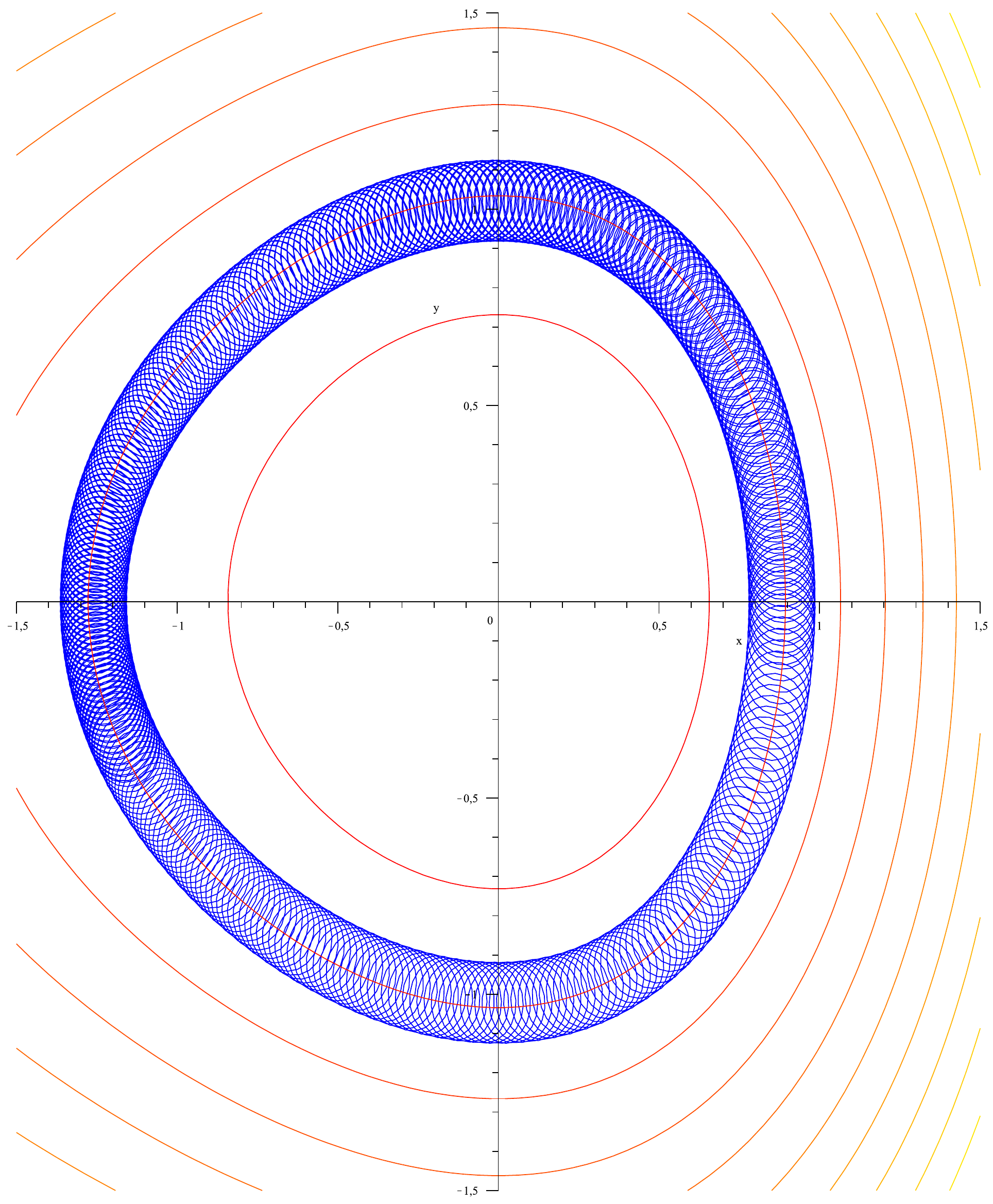}
  \caption{Numerical simulation of the flow of $H$ when the magnetic
    field is given by
    $B(x,y)=2+x^2+y^2+\frac{x^3}3+\frac{x^4}{20}$, and
    $\epsilon=0.05$, $t\in[0,500]$. The picture also displays in red some
    level sets of $B$.}
  \label{fig:numerics}
\end{figure}

\begin{cor}[fast/slow decomposition]
  \label{cor:dynamics}
  Let $N>0$. There exists a small energy $E_0>0$ such that, for all
  $E<E_0$, for times $t\leq T(E)$, the magnetic flow $\phy_H^t$
  at kinetic energy $H=E$ is, up to an error of order
  $\mathcal{O}(E^\infty)$, the Abelian composition of two motions:
  \begin{itemize}
  \item \emph{[fast rotating motion]} a periodic flow around the
    $S^1$-orbits, with frequency $\frac 1{2\pi}\deriv{K}{I}$;
  \item \emph{[slow drift]} the Hamiltonian flow of $h_I$ on
    $\Sigma\simeq\Sigma_I$.
  \end{itemize}
\end{cor}

Thus, we can informally describe the motion as a coupling between a
fast rotating motion around a center $c(t)\in H^{-1}(0)$ and a slow
drift of the point $c(t)$. The rotating motion depends smoothly on
$E$; in terms of the original variables $(q_1,q_2)$, it has a small
radius
\[
r = \frac{E}{B(q)} + \mathcal{O}(E^2)
\]
and a fast angular velocity
\[
\dot\theta = -2B(q) + \mathcal{O}(E).
\]
The motion of $c(t)$, up to an error of order $\mathcal{O}(E^\infty)$,
is given by the effective 1D Hamiltonian $h_I$, depending smoothly on
the adiabatic action $I$, of the form
\[
h_I (x_2,\xi_2) = I{B(q)} + \mathcal{O}(I^2),
\]
where $q$ and $z_2=(x_2,\xi_2)$ are related by~\eqref{equ:phy}.
Notice that, at first order, the flow of $h_I$ is given by the flow of
$IB$; thus, modulo an error of order $E^2$, the trajectories follow
the level sets of the magnetic field; Figure~\ref{fig:numerics} gives
a striking numerical evidence of this.

Under additional hypothesis on $h_I$, one can of course say much
more. For instance, if $h_I$ has no critical points at a given energy
(as in Theorem~\ref{theo:confining2}), then the trajectories are
diffeomorphic to circles; then we can introduce a second adiabatic
invariant. In this case, it could be interesting to improve the
estimates using KAM/Nekhoroshev methods.

We turn now to the quantum counterpart of these results.  Let
$\mathcal{H}_{\h,\A}=(-i\h\nabla-\A)^2$ be the magnetic Laplacian on
$\R^2$, where the potential $\mathbf{A}:\R^2\to\R^2$ is smooth, and
such that $\mathcal{H}_{\h,\A}\in S(m)$ for some order function $m$ on
$\R^4$ (see \cite[Chapter 7]{DiSj99}).  We will work with the Weyl
quantization; for a classical symbol $a=a(x,\xi)\in S(m)$ , it is
defined as:
$$\Op_{\h}^w a\, \psi(x)=\frac{1}{(2\pi\h)^2}\int\int e^{i(x-y)\cdot\xi/\h} a\left(\frac{x+y}{2},\xi\right)\psi(y)\dx y\dx \xi,\quad \forall \psi\in\mathcal{S}(\R^2).$$

The first result shows that the spectral theory of
$\mathcal{H}_{\h,\A}$ is governed at first order by the magnetic field
itself, viewed as a symbol.
\begin{theo}\label{spectrum}
  Assume that the magnetic field $B$ is non vanishing
  ($\Omega=\R^2$). Let $\mathcal{H}^0_\h=\Op_{\h}^w (H^0)$, where
  $H^0=B(\phy^{-1}(z_{2}))|z_{1}|^2$ and the diffeomorphism $\phy$ is
  defined in~\eqref{equ:phy}.  Then there exists a bounded classical
  pseudo-differential operator $Q_\h$, such that
  \begin{itemize}
  \item  $Q_\h$ commutes with $\Op_{\h}^w(\abs{z_1}^2)$;
  \item $Q_\h$ is relatively bounded with respect to
    $\mathcal{H}^0_\h$ with an arbitrarily small relative bound;
  \item its Weyl symbol is $O_{z_2}(\h^2+\h\abs{z_1}^2+\abs{z_1}^4)$,
  \end{itemize}
so that
  the following holds. Assume that the magnetic field is confining:
  there exist constants $\tilde{C}_1>0$, $M_0>0$ such that
  \begin{equation}\label{conf}
  B(q)\geq \tilde{C}_{1} \quad \text{ for } \quad |q|\geq M_{0}.
  \end{equation}
  Let $0<C_1<\tilde{C}_1$. Then the spectra of $\mathcal{H}_{\h,\A}$
  and $\mathcal{N}_{\h}:=\mathcal{H}^0_\h + Q_\h$ in $(-\infty,C_1\h]$
  are discrete. We denote by $0<\lambda_1(\h)\leq \lambda_2(\h)\leq
  \cdots$ the eigenvalues of $\mathcal{H}_{\h,\A}$ and by
  $0<\mu_1(\h)\leq \mu_2(\h)\leq \cdots$ the eigenvalues of
  $\mathcal{N}_{\h}$. Then for all $j\in\N^*$ such that
  $\lambda_j(\h)\leq C_1\h$ and $\mu_j(\h)\leq C_1\h$, we have
\[
\abs{\lambda_j(\h) - \mu_j(\h)} = O(\h^\infty).
\]
\end{theo}

The proof of Theorem \ref{spectrum} relies on the following theorem,
which provides in particular an accurate description of $Q_\h$. In the
statement, we use the notation of Theorem~\ref{theo:classical}; we
recall that $\Sigma$ is the zero set of the classical Hamiltonian $H$.

\begin{theo}\label{main-theo}
  For $\h$ small enough there exists a Fourier Integral Operator
  $U_\h$ such that
\[
U_\h^* U_h = I + Z_\h, \qquad U_\h U_h^* = I + Z'_\h,
\]
where $Z_\h, Z'_\h$ are pseudo-differential operators that
microlocally vanish in a neighborhood of $\tilde\Omega\cap\Sigma$, and
\begin{equation}\label{formal}
  U_\h^* \mathcal{H}_{\h,\A} U_\h = \mathcal{I}_{\h} F_\h +
  R_\h,
 \end{equation}
  where
  \begin{enumerate}
  \item $\mathcal{I}_{\h}:= -\h^2\frac{\partial^2}{\partial x_1^2}
    +x_1^2$;
  \item $F_\h$ is a classical pseudo-differential operator in $S(m)$
    that commutes with $\mathcal{I}_{\h}$;
  \item \label{item:hermite} For any Hermite function $h_n(x_1)$ such that
    $\mathcal{I}_{\h} h_n = \h(2n-1)h_n$, the operator $F^{(n)}_\h$
    acting on $L^2(\R_{x_2})$ by
    \[
    h_n\otimes F^{(n)}_\h (u) = F_\h(h_n\otimes u)
    \]
    is a classical pseudo-differential operator in $S_{\R^2}(m)$ with
    principal symbol
    \[
    F^{(n)}(x_2,\xi_2) = B(q),
    \]
    where $(0,x_2+ i \xi_2)=\phy(q)$ as in~\eqref{equ:phy};
  \item \label{item:R}
    Given any classical pseudo-differential operator
    $D_{\h}$ with principal symbol $d_{0}$ such that $d_{0}( z_{1},
    z_{2})=c( z_{2})| z_{1}|^2+O(| z_{1}|^3)$, and any $N\geq 1$,
    there exist classical pseudo-differential operators $S_{\h,N}$ and
    $K_{N}$ such that:
    \begin{equation}
      R_{\h}=S_{\h,N}(D_{\h})^N+K_{N}+O(\h^{\infty}),
      \label{division}
    \end{equation}

with $K_{N}$ compactly supported away from a fixed neighborhood of
$|z_{1}|=0$.
  \item \label{item:Q} $\mathcal{I}_{\h}F_\h = \mathcal{N}_\h = \mathcal{H}_{\h}^0+
    Q_\h$, where $\mathcal{H}^0_\h=\Op_{\h}^w (H^0)$, $H^0=B(\phy^{-1}(
    z_{2}))|z_{1}|^2$, and the operator $Q_\h$ is relatively bounded
    with respect to $\mathcal{H}^0_\h$ with an arbitrarily small
    relative bound.

  \end{enumerate}
\end{theo}
We recover the result of~\cite{HelKo11}, adding the fact that no odd
power of $\h^{1/2}$ can show up in the asymptotic expansion.
\begin{cor}[Low lying eigenvalues]
\label{cor:low}
  Assume that $B$ has a unique non-degenerate minimum. Then there
  exists a constant $c_0$ such that for any $j$, the eigenvalue
  $\lambda_j(\h)$ has a full asymptotic expansion in integral powers
  of $\h$ whose first terms have the following form:
\[
\lambda_j(\h) \sim \h \min B + \h^2(c_1(2j-1)+c_0) + O(\h^3),
\]
with
$c_1=\frac{\sqrt{\det(B"\circ\phy^{-1}(0))}}{2B\circ\phy^{-1}(0)}$,
where the minimum of $B$ is reached at $\phy^{-1}(0)$.
\end{cor}
\begin{proof}
  The first eigenvalues of $\mathcal{H}_{\h,A}$ are equal to $\h$
  times the eigenvalues of $F_\h^{(1)}$ (in point~\eqref{item:hermite}
  of Theorem~\ref{main-theo}). Since $B$ has a non-degenerate minimum,
  the symbol of $F_\h^{(1)}$ has a non-degenerate minimum, and the
  spectral asymptotics of the low-lying eigenvalues for such a 1D
  pseudo-differential operator are well known. We get
  \[
  \lambda_j(\h) \sim \h \min B + \h^2(c_1(2j-1)+c_0) + O(\h^3),
  \]
  with $c_1=\sqrt{\det(B\circ \phy^{-1})"(0)/2}$. One can easily
  compute
  \[
  c_1 =
  \frac{\sqrt{\det(B"\circ\phy^{-1}(0))}}{2\abs{\det(D\phy^{-1}(0))}}
  = \frac{\sqrt{\det(B"\circ\phy^{-1}(0))}}{2B\circ\phy^{-1}(0)}.
  \]
\end{proof}

Under reasonable assumptions on $B$, Theorems~\ref{main-theo}
and~\ref{spectrum} should yield precise asymptotic expansions even in
the regime of energies larger than $c\h$, where $c>\min B$. For
instance, we obtain the following result.
\begin{cor}[Magnetic excited states]
\label{coro:bs}
  Let $c<\tilde{C}_1$ be a regular value of $B$, and assume that the
  level set $B^{-1}(c)$ is connected. Then there exists $\epsilon>0$
  such that the eigenvalues of the magnetic Laplacian lying in the
  interval $[\h(c-\epsilon), \h(c+\epsilon)]$ have the form
\[
\lambda_j(\h) = (2n-1)\h f_\h(\h n(j),\h k(j)) + O(\h^\infty), \quad (n(j),k(j))\in \Z^2,
\]
where $f_\h=f_0 +\h f_1 + \cdots$ admits an asymptotic expansion in
powers of $\h$ with smooth coefficients $f_i\in\Cinf(\R^2;\R)$ and
$\partial_1 f_0 = 0$, $\partial_2 f_0 \neq 0$. Moreover, the
corresponding eigenfunctions are microlocalized in the annulus
$B^{-1}([c-\epsilon,c+\epsilon])$.

In particular, if $n=1$ and $c\in (\min B, 3\min B)$, the eigenvalues
of the magnetic Laplacian in the interval $[\h(c-\epsilon),
\h(c+\epsilon)]$ have gaps of order {$O(\h^2)$}.

\end{cor}
\begin{proof}
  As before, the spectrum of $\mathcal{H}_{\h,A}$ below $C_1\h$ is the
  union of the eigenvalues below $C_1\h$ of $(2n-1)\h F_\h^{(n)}$,
  $n\in \N^*$. For each $n$, the usual Bohr-Sommerfeld rules for 1D
  semiclassical pseudo-differential operators~(see for
  instance~\cite{san-focus} and the references therein) state that the
  eigenvalues of $F^{(n)}_\h$ in the interval $[c-\epsilon,
  c+\epsilon]$ admit a complete asymptotic expansion of the form
  \[
   f_0^{(n)}(\h j) + \h f_1^{(n)}(\h j) + \cdots,
  \]
  where $f_0^{(n)},f_1^{(n)},\dots$, are smooth functions and
  $f_0^{(n)}=f_0$ does not depend on $n$ and satisfies
  $(f_0^{(n)})'\neq 0$ (precisely, $2\pi f_0^{-1}(c)$ is the area of
  the curve $B^{-1}(c)$ viewed in $\Sigma$, up to a constant).
\end{proof}

\paragraph{Comments on Theorem \ref{main-theo}} When finishing to
write this paper, we discovered that Theorem \ref{main-theo} appears
in a close form in \cite[Theorem 6.2.7]{I98}. However, several
differences have to be mentioned. Our proof uses a deformation
argument \emph{à la} Moser which relies on a global symplectic
parameterization of $\Sigma$ and an intrinsic description of the
symplectic normal bundle $N\Sigma$. Both the classical and quantum
Birkhoff normal forms are obtained simultaneously by endowing the
space of formal series with the semiclassical Weyl product, instead of
the usual product. Actually, the particular grading in $(z_1,\h)$ that
we use is tightly linked to the physical nature of the problem. The
result itself is different since we obtain a uniform remainder $R_\h$
which vanishes to any order in that grading.

\paragraph{Higher dimensions} In~\cite{HelKo13}, the asymptotic
expansion of the eigenvalues is not proved. We believe that the
methods presented in our paper are likely to apply in their context
and should help prove their conjecture.

\vskip 2em

\noindent\emph{Organization of the paper.} The paper is organized as
follows. Section \ref{BNF} is devoted to the proof of Theorems
\ref{theo:classical} and \ref{main-theo}. Then, we prove Theorems
\ref{theo:confining} and \ref{theo:confining2} in Section
\ref{dyna}. Finally in Section \ref{spec} we provide the proof of
Theorem \ref{spectrum}.

\section{Magnetic Birkhoff normal form}\label{BNF}
In this section we prove Theorem \ref{theo:classical}.
\subsection{Symplectic normal bundle of $\Sigma$}
We introduce the submanifold of all particles at rest ($\dot q = 0$):
\[
\Sigma := H^{-1}(0) = \{(q,p); \qquad p = A(q)\}.
\]
Since it is a graph, it is an embedded submanifold of $\R^4$,
parameterized by $q\in\R^2$.
\begin{lem}
  \label{lem:Sigma}
  $\Sigma$ is a symplectic submanifold of $\R^4$. In fact,
  \[
  j^*\omega_{\restr \Sigma} = dA \simeq B,
  \]
  where $j:\R^2\to \Sigma$ is the embedding $j(q)=(q,A(q))$.
\end{lem}
\begin{proof}
  We compute $j^*\omega = j^*(dp_1\wedge dq_1+dp_2\wedge dq_2) =
  (-\deriv{A_1}{q_2}+\deriv{A_2}{q_1}) dq_1\wedge dq_2\neq 0$.
\end{proof}

Since we are interested in the low energy regime, we wish to describe
a small neighborhood of $\Sigma$ in $\R^4$, which amounts to
understanding the normal symplectic bundle of $\Sigma$. For any
$q\in\Omega$, we denote by $T_q\mathbf{A}:\R^2\to\R^2$ the tangent map
of $\mathbf{A}$. Then of course the vectors $(Q,T_q\mathbf{A}(Q))$,
with $Q\in T_q\Omega=\R^2$, span the tangent space
$T_{j(q)}\Sigma$. It is interesting to notice that the symplectic
orthogonal $T_{j(q)}\Sigma^\perp$ is very easy to describe as well.
\begin{lem}
  \label{lem:basis}
  For any $q\in \Omega$, the vectors
  \[
  u_1 := \frac{1}{\sqrt{\abs{B}}}(e_1, \trsp T_q\mathbf{A}(e_1));
  \quad v_1 := \frac{\sqrt{\abs{B}}}{B}(e_2, \trsp T_q\mathbf{A}(e_2))
  \]
  form a symplectic basis of $T_{j(q)}\Sigma^\perp$.
\end{lem}
\begin{proof}
  Let $(Q_1,P_1)\in T_{j(q)}\Sigma$ and $(Q_2,P_2)$ with $P_2=\trsp
  T_q\mathbf{A}(Q_2)$.  Then from~\eqref{equ:omega} we get
  \begin{align*}
    \omega((Q_1,P_1), (Q_2,P_2)) & = \pscal{T_q\mathbf{A}(Q_1)}{Q_2} -
    \pscal{\trsp T_q\mathbf{A}(Q_2)}{Q_1}\\
    & = 0.
  \end{align*}
  This shows that $u_1$ and $v_1$ belong to
  $T_{j(q)}\Sigma^\perp$. Finally
  \begin{align*}
    \omega(u_1,v_1) & = \frac{1}{B}\left( \pscal{\trsp
        T_q\mathbf{A}(e_1)}{e_2} -
      \pscal{\trsp T_q\mathbf{A}(e_2)}{e_1}\right)\\
    & = \frac{1}{B} \pscal{e_1}{(T_q\mathbf{A}-\trsp
      T_q\mathbf{A})(e_2)}\\
    & = \frac{1}{B}\pscal{e_1}{\vec B \wedge e_2} =
    -\frac{B}{B}\pscal{e_1}{e_1} = -1.
  \end{align*}
\end{proof}

Thanks to this lemma, we are able to give a simple formula for the
transversal Hessian of $H$, which governs the linearized (fast)
motion:
\begin{lem}
  The transversal Hessian of $H$, as a quadratic form on
  $T_{j(q)}\Sigma^\perp$, is given by
  \[
  \forall q\in \Omega, \forall (Q,P)\in T_{j(q)}\Sigma^\perp, \quad
  d^2_qH ((Q,P)^2) = 2 \|Q\wedge \vec B\|^2.
  \]
\end{lem}
\begin{proof}
  Let $(q,p)=j(q)$.  From~\eqref{equ:hamiltonian} we get
  \[
  dH = 2 \pscal{p-A}{dp - T_q\mathbf{A}\circ dq}.
  \]
  Thus
  \[
  d^2H((Q,P)^2) = 2 \|(dp - T_q\mathbf{A}\circ dq)(Q,P)\|^2 +
  \pscal{p-A}{M((Q,P)^2)},
  \]
  and it is not necessary to compute the quadratic form $M$, since
  $p-A=0$. We obtain
  \begin{align*}
    d^2H((Q,P)^2) & = 2 \| P - T_q\mathbf{A}(Q) \| ^2\\
    & = 2 \| (\trsp T_q\mathbf{A} - T_q\mathbf{A})(Q) \| ^2 = 2 \| Q
    \wedge \vec B \|^2.
  \end{align*}
\end{proof}
We may express this Hessian in the symplectic basis $(u_1,v_1)$ given
by Lemma~\ref{lem:basis}:
\begin{equation}
  \label{equ:hessian2}
  d^2H _{\restr T_{j(q)}\Sigma^\perp} = 
  \begin{pmatrix}
    2\abs{B} & 0 \\
    0 & 2\abs{B}
  \end{pmatrix}.
\end{equation}
Indeed, $\|e_1\wedge \vec B\|^2= B^2$, and the off-diagonal term is
$\frac{1}{B}\pscal{e_1\wedge \vec B}{e_2 \wedge \vec B} = 0$.

\subsection{Proof of Theorem~\ref{theo:classical}}

We use the notation of the previous section. We endow $\C_{z_1}\times
\R^2_{z_2}$ with canonical variables $z_1=x_1+i\xi_1$,
$z_2=(x_2,\xi_2)$, and symplectic form $\omega_0:=d\xi_1\wedge dx_1 +
d\xi_2\wedge dx_2$. By Darboux's theorem, there exists a
diffeomorphism $g:\Omega\to g(\Omega)\subset\R^2_{z_2}$ such that $g(q_0)=0$
and
\[
g^*(d\xi_2\wedge dx_2) = j^*\omega.
\]
(We identify $g$ with $\phy$ in the statement of the theorem.)  In
other words, the new embedding $\tilde{\jmath}:=j\circ
g^{-1}:\R^2\to\Sigma$ is symplectic. In fact we can give an explicit
choice for $g$ by introducing the global change of variables:
$$x_{2}=q_{1},\quad \xi_{2}=\int_{0}^{q_{2}} B(q_{1},s)\dx s.$$
Consider the following map $\tilde\Phi$:
\begin{align}
  \label{equ:phi}
  \C\times\Omega & \stackrel{\tilde\Phi}\longrightarrow
  N\Sigma \\
  (x_1+i\xi_1, z_2) & \mapsto x_1 u_1(z_2) + \xi_1 v_1(z_2),
\end{align}
where $u_1(z_2)$ and $v_1(z_2)$ are the vectors defined in
Lemma~\ref{lem:basis} with $q=g^{-1}(z_2)$. This is an isomorphism
between the normal symplectic bundle of $\{0\}\times \Omega$ and
$N\Sigma$, the normal symplectic bundle of $\Sigma$: indeed,
Lemma~\ref{lem:basis} says that for fixed $z_2$, the map $z_1\mapsto
\tilde\Phi(z_1,z_2)$ is a linear symplectic map. This implies, by a
general result of Weinstein~\cite{weinstein-symplectic}, that there exists a symplectomorphism $\Phi$ from a neighborhood of
$\{0\}\times\Omega$ to a neighborhood of
$\tilde\jmath(\Omega)\subset\Sigma$ whose differential at
$\{0\}\times\Omega$ is equal to $\tilde\Phi$. Let us recall how to
prove this.

First, we may identify $\tilde\Phi$ with a map into $\R^4$ by
\[
\tilde{\Phi}(z_1,z_2) = \tilde\jmath(z_2) + x_1 u_1(z_2) + \xi_1
v_1(z_2).
\]
Its Jacobian at $z_1=0$ in the canonical basis of $T_{z_1}\C\times
T_{z_2}\Omega=\R^4$ is a matrix with column vectors
$\left[u_1,v_1,T_{z_2}\tilde\jmath(e_1),T_{z_2}\tilde\jmath(e_2)\right]$,
which by Lemma~\ref{lem:basis} is a basis of $\R^4$: thus
$\tilde{\Phi}$ is a local diffeomorphism at every $(0,z_2)$. Therefore
if $\epsilon>0$ is small enough, $\tilde\Phi$ is a diffeomorphism of
$B(\epsilon)\times\Omega$ into its image.

($B(\epsilon)\subset\C$ is the open ball of radius $\epsilon$).

Since $\tilde{\jmath}$ is symplectic, Lemma~\ref{lem:basis} implies
that the basis
$\left[u_1,v_1,T_{z_2}\tilde\jmath(e_1),T_{z_2}\tilde\jmath(e_2)\right]$
is symplectic in $\R^4$; thus the Jacobian of $\tilde\Phi$ on
$\{0\}\times\Omega$ is symplectic. This in turn can be expressed by
saying that the 2-form
\[
\omega_0 - \tilde\Phi^*\omega_{0}
\]
vanishes on $\{0\}\times\Omega$.
\begin{lem}\label{symplectize}
  Let us consider $\omega_{0}$ and $\omega_{1}$ two $2$-forms on
  $\R^4$ which are closed and non degenerate. Let us assume that
  ${\omega_{1}}_{|\hat x_{1}=0}={\omega_{0}}_{|\hat x_{1}=0}$.  There
  exist a neighborhood of $(0,0,0,0)$ and a change of coordinates
  $\psi_{1}$ such that:
$$\psi_{1}^*\omega_{1}=\omega_{0}\quad \mbox{ and }\quad
{\psi_{1}}={\Id} + O(\hat x_1^2).$$
\end{lem}

\begin{proof}
  The proof of this relative Darboux lemma is standard but we recall
  it for completeness (see \cite[p. 92]{McSa98}).

\paragraph{Relative Poincar\'e Lemma}
Let us begin to recall how we can find a $1$-form $\sigma$ defined in
a neighborhood of $\hat z_{1}=0$ such that:
$$\tau:=\omega_{1}-\omega_{0}=d\sigma\quad \mbox{ and
}\quad\sigma=O(\hat x_1^2).$$ We introduce the family of
diffeomorphisms $(\phi_{t})_{0<t\leq 1}$ defined by:
$$\phi_{t}(\hat x_{1}, \hat x_{2}, \hat \xi_{1}, \hat \xi_{2})=(t\hat x_{1}, \hat x_{2}, \hat \xi_{1}, \hat \xi_{2})$$
and we let:
$$\phi_{0}(\hat x_{1}, \hat x_{2}, \hat \xi_{1}, \hat \xi_{2})=(0, \hat x_{2}, \hat \xi_{1}, \hat \xi_{2}).$$
We have:
\begin{equation}
  \phi_{0}^*\tau=0\quad\mbox{ and }\quad \phi_{1}^*\tau=\tau;
  \label{equ:phi_t}
\end{equation}
Let us denote by $X_{t}$ the vector field associated with $\phi_{t}$:
\[
X_{t}=\frac{d\phi_{t}}{dt}(\phi_{t}^{-1})=(t^{-1}\hat x_{1}, 0, 0, 0)
= t^{-1}\hat x_1e_1,
\]
with $e_1:=(1,0,0,0)$.  Let us compute the Lie derivative of $\tau$
along $X_{t}$:
$\frac{d}{dt}\phi_{t}^*\tau=\phi_{t}^*\mathcal{L}_{X_{t}}\tau.$ From
the Cartan formula, we have: $\mathcal{L}_{X_{t}}=
\iota(X_{t})d\tau+d(\iota(X_{t})\tau).$ Since $\tau$ is closed on
$\R^4$, we have $d\tau=0$. Therefore it follows:
\begin{equation}
  \frac{d}{dt}\phi_{t}^*\tau=d(\phi_{t}^*\iota(X_{t})\tau).
  \label{equ:ddt_phi_t}
\end{equation}
We consider the $1$-form $\sigma_{t}:=\phi_{t}^*\iota(X_{t})\tau= \hat
x_1 \tau_{\phi_t(\hat x_{1}, \hat x_{2}, \hat \xi_{1}, \hat
  \xi_{2})}(e_1,\nabla\phi_t(\cdot))=O(\hat x_1^2)$.  We see
from~\eqref{equ:ddt_phi_t} that the map $t\mapsto \phi_t^* \tau$ is
smooth on $[0,1]$.  To conclude, let $\sigma=\int_{0}^1
\sigma_{t}\,dt$; it follows from~\eqref{equ:phi_t}
and~\eqref{equ:ddt_phi_t} that:
$$\frac{d}{dt}\phi_{t}^*\tau=d\sigma_{t}\quad \mbox{ and }\quad \tau=d\sigma.$$
\paragraph{Conclusion}
We use a standard deformation argument due to Moser. For $t\in [0,1]$,
we let: $\omega_{t}=\omega_{0}+t(\omega_{1}-\omega_{0})$. The $2$-form
$\omega_{t}$ is closed and non degenerate (up to choosing a
neighborhood of $\hat z_{1}=0$ small enough). We look for $\psi_{t}$
such that:
$$\psi_{t}^*\omega_{t}=\omega_{0}.$$
For that purpose, let us determine a vector field $Y_{t}$ such that:
$$\frac{d}{dt}\psi_{t}=Y_{t}(\psi_{t}).$$
By using again the Cartan formula, we get:
$$0=\frac{d}{dt}\psi_{t}^*\omega_{t}=\psi_{t}^*\left(\frac{d}{dt}\omega_{t}+\iota(Y_{t})d\omega_{t}+d(\iota(Y_{t})\omega_{t})\right).$$
This becomes:
$$\omega_{0}-\omega_{1}=d(\iota(Y_{t})\omega_{t}).$$
We are led to solve:
$$\iota(Y_{t})\omega_{t}=-\sigma.$$
By non degeneracy of $\omega_{t}$, this determines $Y_{t}$.  Choosing
a neighborhood of $(0,0,0,0)$ small enough, we infer that $\psi_{t}$
exists until the time $t=1$ and that it satisfies
$\psi_{t}^*\omega_{t}=\omega_{0}$. Since $\sigma=O(\hat x_1^2)$, we
get $\psi_{1}=\Id+O(\hat x_{1}^2)$.
\end{proof}

\begin{lem}
  There exists a map $S:B(\epsilon)\times\Omega\to
  B(\epsilon)\times\Omega$, which is tangent to the identity along
  $\{0\}\times\Omega$, such that
  \[
  S^* \tilde\Phi^*\omega = \omega_0.
  \]
\end{lem}
  
\begin{proof}
  It is sufficient to apply Lemma \ref{symplectize} to
  $\omega_{1}=\tilde\Phi^*\omega_{0}$.
\end{proof}

We let $\Phi:=\tilde\Phi\circ S$; this is the claimed symplectic map.

Let us now analyze how the Hamiltonian $H$ is transformed under
$\Phi$. The zero-set $\Sigma=H^{-1}(0)$ is now $\{0\}\times\Omega$,
and the symplectic orthogonal
$T_{\tilde\jmath(0,\hat{z}_2)}\Sigma^\perp$ is canonically equal to
$\C\times\{\hat{z}_2\}$. By~\eqref{equ:hessian2}, the matrix of the
transversal Hessian of $H\circ\Phi$ in the canonical basis of $\C$ is
simply
\begin{equation}
  \label{equ:hessian3}
  d^2(H\circ\Phi)_{\restr \C\times\{\hat z_2\}} = d^2_{\Phi(0,\hat z_2)}H\circ (d\Phi)^2 =  
  \begin{pmatrix}
    2\abs{B(g^{-1}(\hat z_2))} & 0 \\
    0 & 2\abs{B(g^{-1}(\hat z_2))}
  \end{pmatrix}.
\end{equation}
Therefore, by Taylor's formula in the $\hat z_1$ variable (locally
uniformly with respect to the $\hat z_2$ variable seen as a
parameter), we get
\begin{align*}
  H\circ\Phi(\hat z_1,\hat z_2) & = H\circ\Phi_{\restr \hat z_1=0} +
  dH\circ\Phi_{\restr \hat z_1=0}(\hat z_1) + \frac{1}{2}
  d^2(H\circ\Phi)_{\restr \hat z_1=0}(\hat z_1^2) + \mathcal{O}(\abs{\hat z_1}^3)\\
  & = 0 + 0 +\abs{B(g^{-1}(\hat z_2))}\abs{\hat z_1}^2 +
  \mathcal{O}(\abs{\hat z_1}^3).
\end{align*}
In order to obtain the result claimed in the theorem, it remains to
apply a formal Birkhoff normal form in the $\hat z_1$ variable, to
simplify the remainder $\mathcal{O}(\hat z_1^3)$.  This classical
normal form is a particular case of the semiclassical normal form that
we prove below (Proposition~\ref{prop:formal-normal-form}); therefore
we simply refer to this proposition, and this finishes the proof of
the theorem, where, for simplicity of notation, the variables $(z_1,z_2)$ actually
refer to $(\hat{z}_1,\hat{z}_2)$.

\subsection{Semiclassical Birkhoff normal form}
We follow the spirit of \cite{VuCha08, Vu09}. In the coordinates $\hat
x_{1}, \hat\xi_{1}, \hat x_{2}, \hat\xi_{2}$, the Hamiltonian takes
the form:
\begin{equation}\label{DL3}
  \hat{H}(\hat z_{1}, \hat z_{2})=H^0+O(|\hat z_{1}|^3),\quad\mbox{ where } H^0=B(g^{-1}(\hat z_{2}))|\hat z_{1}|^2.
\end{equation}
Let us now consider the space of the formal power series in $\hat
x_{1}, \hat \xi_{1}, \h$ with coefficients smoothly depending on
$(\hat x_{2},\hat \xi_{2})$ : $\mathcal{E}=\mathcal{C}^\infty_{\hat
  x_{2}, \hat\xi_{2}}[\hat x_{1}, \hat\xi_{1},\h]$. We endow
$\mathcal{E}$ with the Moyal product (compatible with the Weyl
quantization) denoted by $\star$ and the commutator of two series
$\kappa_{1}$ and $\kappa_{2}$ is defined as:
$$[\kappa_{1},\kappa_{2}]=\kappa_{1}\star \kappa_{2}-\kappa_{2}\star\kappa_{1}.$$

\begin{notation}
  The degree of $\hat x_{1}^\alpha\hat\xi_{1}^\beta \h^l$ is
  $\alpha+\beta+2l$. $\mathcal{D}_{N}$ denotes the space of the
  monomials of degree $N$. $\mathcal{O}_{N}$ is the space of formal
  series with valuation at least $N$.
\end{notation}

\begin{prop}\label{prop:formal-normal-form}
  Given $\gamma\in\mathcal{O}_{3}$, there exist formal power series
  $\tau,\kappa\in\mathcal{O}_{3}$ such that:
$$e^{i\h^{-1}\ad_{\tau}}(H^0+\gamma)=H^0+\kappa,$$
with: $[\kappa,H^0]=0.$
\end{prop}
\begin{proof}
  Let $N\geq 1$. Assume that we have, for $N\geq 1$ and
  $\tau_{N}\in\mathcal{O}_{3}$:
$$e^{ih^{-1}\ad_{\tau_{N}}}(H^0+\gamma)=H^0+K_{3}+\cdots+K_{N+1}+R_{N+2}+\mathcal{O}_{N+3},$$
where $K_{i}\in\mathcal{D}_{i}$ commutes with $|\hat z_{1}|^2$ and
where $R_{N+2}\in\mathcal{D}_{N+2}$.

Let $\tau'\in \mathcal{D}_{N+2}$. A computation provides:
\begin{align*}
  &e^{ih^{-1}\ad_{\tau_{N}+\tau'}}(H^0+\gamma)=H^0+K_{3}+\cdots+K_{N+1}+K_{N+2}+\mathcal{O}_{N+3},
\end{align*}
with:
$$K_{N+2}=R_{N+2}+B(g^{-1}(\hat z_{2}))i\h^{-1}\ad_{\tau'} |\hat z_{1}|^2=R_{N+2}-B(g^{-1}(\hat z_{2}))i\h^{-1}\ad_{|\hat z_{1}|^2} \tau'.$$
We can write:
$$R_{N+2}=K_{N+2}+B(g^{-1}(\hat z_{2}))i\h^{-1}\ad_{|\hat z_{1}|^2} \tau'.$$
Since $B(g^{-1}(\hat z_{2}))\neq 0$, we deduce the existence of
$\tau'$ and $K_{N+2}$ such that $K_{N+2}$ commutes with $|\hat
z_{1}|^2$.  Note that $i\h^{-1}\ad_{|\hat z_{1}|^2}=\{|\hat
z_{1}|^2,\cdot\}$.
\end{proof}

\subsection{Proof of Theorem \ref{main-theo}}

Since the formal series $\kappa$ given by Proposition
\ref{prop:formal-normal-form} commutes with $H^0$, we can write it
as a polynomial in $|\hat z_{1}|^2$:
$$\kappa=\sum_{k\geq 0} \sum_{2l+m=k} \h^l c_{l,m}(\hat z_{2}) |\hat z_{1}|^{2m}.$$
This formal series can be reordered by using the monomials $(|\hat
z_{1}|^2)^{\star m}$ for the product law $\star$:
$$\kappa=\sum_{k\geq 0} \sum_{2l+m=k} \h^l c^\star_{l,m}(\hat z_{2})(|\hat z_{1}|^2)^{\star m}.$$
Thanks to the Borel lemma, there exists a smooth function with compact
support $f^\star(\h,|\hat z_{1}|^2,\hat z_{2})$ such that the Taylor
expansion with respect to $(\h,|\hat z_{1}|^2)$ of $f^\star(\h,|\hat
z_{1}|^2,\hat z_{2})$ is given by $\kappa$ and:
\begin{equation}\label{f-star}
  \sigma^{\T,w}\left(\Op_{\h}^w \left(f^{\star}(\h, \mathcal{I}, z_{2})\right)\right)=\kappa,
\end{equation}
where $\sigma^{\T,w}$ means that we consider the formal Taylor series
of the Weyl symbol with respect to $(\h,\hat z_{1})$. The operator
$\Op_{\h}^w \left(f^{\star}(\h,\mathcal{I}, z_{2})\right)$ has to be
understood as the Weyl quantization with respect to $\hat z_{2}$ of an
operator valued symbol. 
We can write it in the form:
$$\Op_{\h}^w f^\star(\h,\mathcal{I}_{\h},\hat z_{2})=\mathcal{I}_{\h}\Op_{\h}^w \tilde f^\star(\h,\mathcal{I}_{\h},\hat z_{2})$$
so that, up to choosing the support of $f^{\star}$ small enough, there
exists $\eta_{0}$ such that for $\eta\in(0,\eta_{0})$, we have, for
all $\psi\in\Cinf_0(\R^2)$, 
\begin{equation}
  |\langle\Op_{\h}^w f^\star(\h,\mathcal{I}_{\h}, \hat z_{2})\psi,
  \psi\rangle|\leq \eta \|\mathcal{I}_{\h}^{1/2} \psi\|^2.
  \label{equ:relative}
\end{equation}

Moreover we can also introduce a smooth symbol
$a_{\h}$ with compact support such that
$\sigma^{\T,w}(a_{\h})=\tau$. Using \eqref{DL3} and applying the
Egorov theorem (see \cite[Theorems 5.5.5 and 5.5.9]{Martinez02}, \cite{Ro87} or \cite{Z13}), we can find a microlocally unitary Fourier Integral
Operator $V_{\h}$ such that:
$$V_{\h}^* \mathcal{H}_{\h,\A}V_{\h}=C_{0}\h+ \mathcal{H}^0+\Op_{\h}^w(r_{\h}),\quad \mbox{ with } \mathcal{H}^0=\Op_{\h}^w (H^0)$$
so that
$\sigma^{\T,w}\left(\Op_{\h}^w(r_{\h})\right)=\gamma\in\mathcal{O}_{3}$. In
fact, one can choose $V_\h$ such that the subprincipal symbol is
preserved by conjugation (see for instance \cite[Appendix A]{HelSj89}), which implies that
$C_0=0$\footnote{We give another proof of this fact in
  Remark~\ref{rem:C0} below.}. It remains to use Proposition
\ref{prop:formal-normal-form} and again the Egorov theorem to notice
that $e^{i\h^{-1}\Op_{\h}^w(a_{\h})} \Op_{\h}^w(r_{\h})
e^{-i\h^{-1}\Op_{\h}^w(a_{\h})}$ is a pseudo-differential operator
such that the formal Taylor series of its symbol is $\kappa$.
Therefore, recalling \eqref{f-star}, we have found a microlocally
unitary Fourier Integral Operator $U_{\h}$ such that:
\begin{equation}\label{series}
  U_{\h}^*\mathcal{H}_{\h,\A} U_{\h}= \mathcal{H}^0+\Op_{\h}^w \left(f^{\star}(\h,\mathcal{I},z_{2})\right)+R_{\h},
\end{equation}
where $R_{\h}$ is a pseudo-differential operator such that
$\sigma^{\T,w}(R_{\h})=0$. It remains to prove the division property
expressed in the last statement of item~\eqref{item:R} of
Theorem~\ref{main-theo}. By the Morse Lemma, there exists in a fixed
neighborhood of $z_1=0$ in $\R^4$ a (non symplectic) change of
coordinates $\tilde{z}_1$ such that
$d_0=c(z_2)\abs{\tilde{z}_1}^2$. It is enough to prove the result in
this microlocal neighborhood.  Now, for any $N\geq 1$, we proceed by
induction.  We assume that we can write $R_{\h}$ in the form:
$$R_{\h}=\Op_{\h}^w\left(s_{0}+\h s_{1}+\cdots+\h^k s_{k}\right)D_{\h}^N+O(\h^{k+1}),$$
with symbols $s_{j}$ which vanish at infinite order with respect to
$\hat z_{1}$. We look for $s_{k+1}$ such that:
$$R_{\h}=\Op_{\h}^w\left(s_{0}+\h s_{1}+\cdots+\h^k s_{k}+\h^{k+1}s_{k+1}\right)D_{\h}^N+O(\h^{k+2})\tilde R_{\h,k}.$$
We are reduced to find $s_{k+1}$ such that:
$$\tilde r_{0,k}=d_{0}^N s_{k+1}.$$
Since $\tilde r_{0,k}$ vanishes at any order at zero we can find a
smooth function $\phi_k$ such that:
$$\tilde r_{0,k}=|\tilde z_{1}|^{2N}\phi.$$
We have $s_{k+1}(\tilde z_1,z_2)=\frac{\phi_k(\tilde z_1,z_2)}{c(z_2)^N}$.

This ends the proof of Theorem \ref{main-theo}.
\begin{rem}
\label{rem:C0}
  It is well known that (see \cite[Theorem 1.1]{HelMo01}), when $B > 0$, the smallest eigenvalue
  $\lambda_1(\h)$ of $\mathcal{H}_{\h,A}$ has the following
  asymptotics
\[
\lambda_1(\h) \sim \h \min_{q\in\R^2} B(q).
\]
We will see in Section~\ref{sec:local-micr-eigenf} that the
corresponding eigenfunctions are microlocalized on $\Sigma$ at the
minima of $B$. Therefore the normal form would imply, by a
variational argument, that
\begin{equation}
  \lambda_1(\h) \geq C_0\h + \mu_1(\h) + o(\h),\label{equ:lambda1}
\end{equation}
where $\mu_1(\h)$ is the smallest eigenvalue of
$\mathcal{N}_\h:=\mathcal{H}^0+\Op_{\h}^w
\left(f^{\star}(\h,\mathcal{I},z_{2})\right)$. Similarly, we will see
in~\ref{sec:micr-eigenf-N} that the lowest eigenfunctions of
$\mathcal{N}_\h$ are also microlocalized in $\hat{z_1}$ and
$\hat{z_2}$, and therefore 
\[
\lambda_1(\h) \sim  C_0\h + \mu_1(\h).
\]
By G\aa{}rding's inequality and point~\eqref{item:Q} of
Theorem~\ref{main-theo}, $\mu_1(\h)\sim \h\min B$. Comparing with
\eqref{equ:lambda1}, we see that $C_0=0$.
\end{rem}

\section{Long time dynamics at low energy}\label{dyna}

The goal of this section is to prove Theorems~\ref{theo:confining} and
\ref{theo:confining2}. We shall rely on the
following localization lemma.

We work in the open set $\Omega$ equipped with the coordinates
$(z_1,z_2)$ given by the normal form of Theorem~\ref{theo:classical};
thus, we may write $H(z_1,z_2)=K+R$, where
$K=\abs{z_1}^2f(z_2,\abs{z_1}^2)$ and the Taylor series of $R$ with
respect to $z_1$ vanishes for all $z_2$. On $\Omega$, the magnetic
field $B$ has a fixed sign. For notational simplicity we may assume
that $B>0$. We denote by $\phy_H^t$ the Hamiltonian flow of $H$,
$I=\abs{z_1}^2$, and $I(t):=I\circ \phy_H^t$. We also denote
$z_2(t):=z_2\circ\phy_H^t$.
\begin{lem}\label{lem:classical-confining}
  Let $C_f>0$, $M>0$ be such that
  \begin{equation}
    f(z_2,0) > C_f, \qquad \forall \abs{z_2} >M.
    \label{equ:confining_f}
  \end{equation}
  Let $0<\tilde c_0 < c_0<C_0< \tilde C_0 < C_f$.  For any $\epsilon>0$ we define the bounded open set
  \begin{equation}
    U_\epsilon := \left\{(z_1,z_2); \qquad \abs{z_1}^2< \frac\epsilon{2},
      \quad c_0< f(z_2,0) < C_0 \right\},
  \label{equ:U}
  \end{equation}
  which is contained in the compact set
   \begin{equation}
    V_\epsilon := \left\{(z_1,z_2); \qquad \abs{z_1}^2\leq \epsilon,
      \quad \tilde c_0\leq f(z_2,0) \leq \tilde C_0 \right\}.
  \end{equation}

  Let
  \[
  T_\epsilon := \sup \{ T>0; \quad \forall t\in[-T,T], \phy_H^{t} \text{ exists and }
  (z_1(t),z_2(t))\in V_{\eps}\text{ for any starting point in }
  U_\epsilon \}.
  \]
  Then for any $N>0$, there exists $\epsilon_0>0$ and a constant $C>0$
  such that
  \[
  \forall \epsilon \leq \epsilon_0, \qquad T_\epsilon \geq
  \frac{C}{\epsilon^N}.
  \]
\end{lem}
\begin{proof}
  Let $N>1/2$.  Since $V_{\epsilon}$ is compact,  we have $T_{\eps}>0$; moreover, there exists
  $\epsilon_0$ such that $U_\epsilon\subset \Omega$ for all
  $\epsilon\leq \epsilon_0$.  Since the $z_1$-Taylor series of $R$
  vanishes, we can write $R=I^NR_N$, where $R_N$ is smooth. Thus
  \[
  \{H,I\} = I^N\{R_N,I\},
  \]
  which implies
  \[
  \abs{\{H,I\}} \leq 2 I^{N+1/2} \norm{\nabla R_N}.
  \]
  Therefore, we get, on $U_\epsilon$,
  \[
  \forall \abs{t}<T_\epsilon, \quad \abs{\frac{d}{dt}I(t)} =
  \abs{\{H,I\}\circ\phy_H^t}\leq 2 C_N I(t)^{N+1/2},
  \]
  where $C_N:=\sup_{V_{\epsilon_0}}\norm{\nabla R_N}$. By integration,
  we get
  \begin{equation}
    \label{equ:I}
    \forall \abs{t}<T, \qquad  \abs{I(t) - I(0)} \leq 2C_N\abs{t}\epsilon^{N+1/2}.
  \end{equation}

We apply a similar argument for $K(t):=K\circ\phy_H^t$. We have
$\{H,K\}= \{I^NR_N,K\}= I^N\{R_N,K\}$. Thus we get, on $U_\epsilon$,
\[
\abs{\frac{d}{dt}K(t)} \leq I^{N+1/2} C'_N,
\]
with $C'_N:=3\sup_{V_{\epsilon_0}}\abs{\{R_N,K/I\}}$. Therefore
$\abs{K(t)-K(0)}\leq C'_N I^{N+1/2}\abs{t}$, which implies, since
$K=If(z_2,I)$,
\begin{equation}
  \abs{f(z_2(t),I(t)) - f(z_2(0),I(0))} \leq  C'_{N} I^{N-1/2}\abs{t} \leq  C'_{N} \epsilon^{N-1/2}\abs{t}.
  \label{equ:f}
\end{equation}

We may write $f(z_2,I)=f(z_2,0) + I\tilde{f}$, for some smooth
function $\tilde{f}$. 

 Let us fix $\ell>0$ such that $C_{0}+\frac{C_{f}-C_{0}}{\ell}<\tilde C_{0}$ and $c_{0}-\frac{C_{f}-C_{0}}{\ell}>\tilde c_{0}$. 
Assume that $\epsilon_0$ is small enough so that
\begin{equation}
 \sup_{V_{\epsilon_0}}\tilde{f} \leq (C_f-C_0)/(2\ell \epsilon_0).
  \label{equ:sup_f}
\end{equation}

Assume by contradiction that there exists $\epsilon\leq \epsilon_0$
such that
\begin{equation}
  C' _{N} \epsilon^{N-1/2} T_\epsilon \leq (C_f-C_0)/(2\ell),
  \label{equ:contra_f}
\end{equation}
and
\begin{equation}
  2C_N \epsilon^{N+1/2} T_\epsilon \leq \epsilon/3.
  \label{equ:contra_I}
\end{equation}
By~\eqref{equ:U},  Equations~\eqref{equ:contra_I} and \eqref{equ:I}
 imply $I(t)\leq \epsilon/2 + \epsilon/3 = 5\epsilon/6$. 
Equations~\eqref{equ:contra_f} and \eqref{equ:f}  imply $
f(z_2(t),I(t)) \leq C_0 + (C_f-C_0)/(2\ell)$, and hence,
by~\eqref{equ:sup_f},
\[
 f(z_2(t),0) \leq \tilde C_{0}-\left(\tilde C_{0}-C_{0}-\frac{C_{f}-C_{0}}{\ell}\right)=\tilde C_{1}<\tilde C_{0}<C_f.
\]
 In the same way we find 
 \[f(z_2(t),0) \geq c_{0}-\frac{C_{f}-C_{0}}{\ell}=\tilde c_{1}>
 \tilde c_{0}.\] Now we remark that, by definition of $T_\epsilon$,
 the flow $\phy_H^t$ is uniformly bounded for all
 $\abs{t}<T_\epsilon$; therefore, there exists $T'>T_\epsilon$ for
 which the flow $\phy_H^t$ is defined for all $\abs{t}<T'$. Since
 $I(t)\leq 5\epsilon/6$ and $\tilde c_{1}\leq f(z_2(t),0) \leq \tilde
 C_{1}$ for all $t<T_\epsilon$, we can find $T'>T_\epsilon$ such that
 $ z(t)\in V_{\eps}$ which contradicts the definition of $T_\epsilon$.

Therefore one of \eqref{equ:contra_f} or~\eqref{equ:contra_I} must be
false. In both cases, we find a constant $C>0$ such that
\[
\forall \epsilon< \epsilon_0, \qquad T_\epsilon \geq
\frac{C}{\epsilon^{N-1/2}},
\]
which gives the result.
\end{proof}
\subsection{Proof of Theorems~\ref{theo:confining} and~\ref{theo:confining2}}
The confining assumption on $B$ implies~\eqref{equ:confining_f} ---
with different constants. Hence, we may apply
Lemma~\ref{lem:classical-confining} to $H$ and $K$ which implies that
the flows $\phy_H^t$ and $\phy_K^t$ remain  in the region
$V_{\eps}$  for times of order
$\epsilon^{-N}$, and starting position in $U_\epsilon$. This proves
the first point of Theorem~\ref{theo:confining}.

Now, let $N'>N$. Writing $H=K+I^{N'}R_{N'}$, we see that the
Hamiltonian vector fields $\ham{H}$ and $\ham{K}$ differ by
$\mathcal{O}(\epsilon^{N'-1/2})$. Let $\mathcal{F}(t) = \phy_H^t - \phy_K^t$;
$d\mathcal{F}/dt = \ham{H}\circ\phy_H^t-\ham{K}\circ\phy_K^t$. By Taylor, we get
\[
\frac{d\mathcal{F}}{dt} = \ham{H-K}\circ \phy_H^t + \mathcal{O}(\phy_K^t -
\phy_H^t),
\]
where the $\mathcal{O}$ is given by the derivatives of $\ham{K}$ and
thus is uniform for $\abs{t}< T_\epsilon$. Thus there exist constants
$C_1,C_2$ such that
\[
\norm{\frac{d\mathcal{F}}{dt}} \leq
C_1\epsilon^{N'-1/2} + C_2\norm{\mathcal{F}(t)}.
\]
Here we use $\norm{\cdot}$ for the Euclidean norm in $\R^4$.
Therefore, since $\mathcal{F}(0)=0$,  the Gronwall lemma provides
\[
\norm{\mathcal{F}(t)} \leq \frac{C_1\epsilon^{N'-1/2}}{C_2}(e^{C_2\abs{t}}-1).
\]
Thus, if $\abs{t}\leq C_3\abs{\ln\epsilon}$ we get
$\norm{\mathcal{F}(t)}\leq C_1C_2^{-1}\epsilon^{N'-1/2-C_2C_3}$, which
proves~\eqref{equ:two-flows} since $N'$ is arbitrary, thereby
establishing Theorem~\ref{theo:confining}.

The naive estimate used above in the proof of
Theorem~\ref{theo:confining} cannot yield the stronger conclusion of
Theorem~\ref{theo:confining2}, because it does not take into account
the commutation $\{H,K\}=0$. For this we consider the composition
$\phy_K^t\circ\phy_H^{-t}$.  Notice that, thanks to
Lemma~\ref{lem:classical-confining}, $\phy_H^{-t}$ sends
$U_{\epsilon}$ into $V_\epsilon$ for times of order
$\mathcal{O}(\epsilon^{-N})$, and that $V_\epsilon$ is globally
invariant by $\phy_K^t$ for all times. Thus, the composition
$\phy_K^t\circ\phy_H^{-t}$ is well defined on $U_{\epsilon}$ and takes
values in $V_\epsilon$, for times of order
$\mathcal{O}(\epsilon^{-N})$.

Let us fix an arbitrary smooth function $z : \R^4\to\R$ and introduce,
on $U_{\epsilon}$, the family of functions
\[
\mathcal{D}(t):=z\circ\phy_K^t\circ\phy_H^{-t}.
\]
Using, among others, the equivariance of the Poisson bracket under
symplectomorphism, we get
\[
\frac{d\mathcal{D}(t)}{dt} = -\{H,\mathcal{D}\} +
\{K\circ\phy_H^{-t},\mathcal{D}\} = \{-R\circ\phy_H^{-t},\mathcal{D}\}
= -\{R,z\circ\phy_K^t\}\circ\phy_H^{-t}.
\]
The goal is now to estimate $\{R,z\circ\phy_K^t\}$ on
$ V_\epsilon$. We have 
\[
\ham{K} = f\ham{I} + I\ham{f(z_2,I)}, 
\]
and since $\{I,f(z_2,I)\}=0$, the flow of $K$ can be written as
\[
\phy_K^t = \phy_I^{tf}\circ \phy_{f(z_2,I)}^{It},
\]
and $I$ is constant along this flow. We use now the assumptions of
Theorem~\ref{theo:confining2}; thus, $d_{z_2}f(z_2,0)$ does not vanish
on the annulus $c_0\leq f(z_2,0)\leq C_0$, where $J=[c_0,C_0]$. This
implies that the same holds for $d_{z_2}f(z_2,I)$, when $I<\epsilon_0$
is small enough. Therefore, for each value of $I$ one can apply the
action-angle theorem to the Hamiltonian $z_2\mapsto f(z_2,I)$: there
exists a symplectic change of coordinates $(r,\theta)=\psi_I(z_2)$,
with $(r,\theta)\in \R\times S^1$, such that
\[
\phy_{f(z_2,I)}^t(r,\theta) = (r, \theta+tg(I,f)),
\]
where $g$ is smooth.
This yields the following formula for the flow of $K$ in the variables
$(z_1,r,\theta)$:
\[
\phy_K^t(z_1,r,\theta) = (e^{-2itf}z_1,r,\theta+tIg(I,f)).
\]
From this we obtain the estimate for the spacial derivative:
\[
\norm{d\phy_K^t} \leq C(\abs{t}+1) \qquad \text{ on } V_\epsilon,
\]
for some constant $C>0$ (involving the $C^1$-norms of $f$ and $g$ on
$ V_\epsilon$), and for any $t\in\R$. Now, as above, we write
$R=I^{N'}R_{N'}$ and get
\[
\{R,z\circ\phy_K^t\} = I^{N'}\{R_{N'},z\circ\phy_K^t\} + N'R_{N'}I^{N'-1}\{I,z\circ\phy_K^t\},
\]
hence
\[
\abs{\{R,z\circ\phy_K^t\} } \leq C_{N'} I^{N'-1/2}\norm{d\phy_K^t} \leq
\tilde{C}_{N'} I^{N'-1/2} (1+\abs{t}).
\]
Thus, if $\abs{t}\leq T_\epsilon = \mathcal{O}(\epsilon^{-N})$, we
obtain, on $ U_{\epsilon}$,
\[
\abs{\mathcal{D}(t)-\mathcal{D}(0)} \leq \hat{C}_{N,N'} I^{N'-N-1/2}.
\]
Taking $z$ to be any coordinate function, we get, for $m\in U_{\epsilon}$,
\[
\norm{\phy_K^t\circ\phy_H^{-t}(m)-m}\leq C_{N,N'} \epsilon^{N'-N-1/2}.
\]
 Notice that an estimate of the same kind is also valid for $m\in\overset{\circ}{V}_{\eps}$.  For any $m'\in U_{\epsilon}$ we may let $m=\phy_H^t(m')$, which yields
\[
\norm{\phy_K^t(m')-\phy_H^t(m')}\leq C_{N,N'} \epsilon^{N'-N-1/2}.
\]
This gives the conclusion of Theorem~\ref{theo:confining2} by
choosing $N'$ large enough.

\section{Spectral theory}\label{spec}
This section is devoted to the proof of Theorem \ref{spectrum}. The
main idea is to use the eigenfunctions of $\mathcal{H}_{\h,\A}$ and
$\mathcal{N}_{\h}$ as test functions in the pseudo-differential
identity \eqref{formal} given in Theorem \ref{main-theo} and to apply
the variational characterization of the eigenvalues given by the
min-max principle. In order to control the remainders we shall prove
the microlocalization of the eigenfunctions of $\mathcal{H}_{\h,\A}$
and $\mathcal{N}_{\h}$ thanks to the confinement assumption
\eqref{conf}.

\subsection{Localization and microlocalization of the eigenfunctions
  of $\mathcal{H}_{\h,\A}$}
\label{sec:local-micr-eigenf}
The space localization of the eigenfunctions of $\mathcal{H}_{\h,\A}$,
which is the quantum analog of Theorem \ref{theo:confining}, is a
consequence of the so-called Agmon estimates.
\begin{prop}\label{Agmon}
  Let us assume \eqref{conf}.
  Let us fix $0<C_{1}<\tilde C_{1}$ and $\alpha\in(0,\frac12)$. There
  exist $C,\h_{0},\eps_0>0$ such that for all $0<\h\leq\h_0$ and for all eigenpair $(\la, \psi)$
  of $\mathcal{H}_{\h,\A}$ such that $\la\leq C_{1}\h$, we have:
$$\int |e^{\chi(q)\h^{-\alpha}|q|} \psi|^2\dx q\leq C\|\psi\|^2,$$
where $\chi$ is zero for $|q|\leq M_{0}$ and $1$ for $|q|\geq M_0 +
\eps_{0}$. Moreover, we also have the weighted $H^1$ estimate:
$$\int |e^{\chi(q)\h^{-\alpha}|q|} h\nabla\psi|^2 \dx q\leq C\h\|\psi\|^2.$$
\end{prop}
\begin{proof}
  Let us denote by $q_{\h,\A}$ the quadratic form associated with the
  magnetic Laplacian $\mathcal{H}_{\h,\A}$. We write the Agmon formula
  (see \cite{Agmon82, Agmon85}) for the eigenpair $(\la,\psi)$ with
  $\la\leq C_{1}\h$:
$$q_{\h,\A}(e^{\Phi}\psi)=\la\|e^{\Phi}\psi\|+\h^2\|\nabla\Phi e^{\Phi}\|^2.$$
We recall that:
$$q_{\h,\A}(e^{\Phi}\psi)\geq \int \h B(q)|e^{\Phi}\psi|^2\dx q$$
so that:
$$\int \left(\h B(q)-C_{1}\h-\h^2\|\nabla\Phi\|^2\right)|e^\Phi \psi|^2\dx q\leq 0.$$
We split this integral into two parts:
$$\int_{|q|\geq M_0}
\left(\h B(q)-C_{1}\h-\h^2\|\nabla\Phi\|^2\right)|e^\Phi \psi|^2\dx
q\leq \int_{|q|\leq M_{0}} \left(-\h
  B(q)+C_{1}\h+\h^2\|\nabla\Phi\|^2\right)|e^\Phi \psi|^2\dx q.$$ Let
us choose $\Phi$:
$$\Phi=\chi(q)\h^{-\alpha}|q|.$$
We get:
$$\int_{|q|\geq M_{0}} \left(\h B(q)-C_{1}\h-\h^2\|\nabla\Phi\|^2\right)|e^\Phi \psi|^2\dx q\leq Ch\|\psi\|^2.$$
Then we have:
$$\int_{|q|\geq M_{0}} \left(\h C_{1}-C_{1}\h-\tilde C \h^{2-2\alpha}\right)|e^\Phi \psi|^2\dx q\leq Ch\|\psi\|^2.$$
We infer that:
$$\int_{|q|\geq M_{0}}|e^\Phi \psi|^2\dx q\leq C\|\psi\|^2,\quad \int |e^{\chi(q)\h^{-\alpha}|q|} \psi|^2\dx q\leq C\|\psi\|^2$$
and then:
$$q_{h,\A}(e^{\Phi}\psi)\leq C\h\|\psi\|^2.$$
\end{proof}
\begin{rem}
  This estimate is interesting when $|x|\geq M_0+ \eps_{0}$. In this
  region, we deduce by standard elliptic estimates that
  $\psi=O(h^{\infty})$ in suitable norms (see for instance
  \cite[Proposition 3.3.4]{Hel88} or more recently~\cite[Proposition
  2.6]{Ray12}). Therefore, the eigenfunctions are localized in space
  in the ball of center $(0,0)$ and radius $M_0+\eps_0$.
\end{rem}
We shall now prove the microlocalization of the eigenfunctions near
the zero set of the magnetic Hamiltonian $\Sigma$.
\begin{prop}\label{micro-loc-L}
  Let us assume \eqref{conf}. Let us fix $0<C_{1}<\tilde C_{1}$ and
  consider $\delta\in\left(0,\frac{1}{2}\right)$. Let $(\la,\psi)$ be
  an eigenpair of $\mathcal{H}_{\h,\A}$ with $\la\leq C_{1}\h$. Then,
  we have:
$$\psi=\chi_{1}\left(\h^{-2\delta}\mathcal{H}_{\h,\A}\right)\chi_{0}(q)\psi+O(\h^{\infty}),$$
where $\chi_{0}$ is smooth cutoff function supported in a compact set
in the ball of center $(0,0)$ and radius $M_{0}+\eps_{0}$ and where $\chi_{1}$ a smooth cutoff function being $1$ near
$0$.
\end{prop}
\begin{proof}
In view of Proposition~\ref{Agmon}, it is enough to prove that
  \begin{equation}\label{chi-chi0}
    \left(1-\chi_{1}\left(\h^{-2\delta}\mathcal{H}_{\h,\A}\right)\right) (\chi_{0}(q)\psi)=O(h^{\infty}).
  \end{equation}
  By the space localization, we have:
$$\mathcal{H}_{\h,\A} (\chi_{0}(q)\psi)=\la \chi_{0}(q)\psi+O(\h^{\infty}).$$
Then, we get:
$$   \left(1-\chi_{1}\left(\h^{-2\delta}\mathcal{H}_{\h,\A}\right)\right)\mathcal{H}_{\h,\A} (\chi_{0}(q)\psi)=\la   \left(1-\chi_{1}\left(\h^{-2\delta}\mathcal{H}_{\h,\A}\right)\right)\left( \chi_{0}(x)\psi\right)+O(\h^{\infty}).$$
This implies:
\begin{multline*}
  \h^{2\delta}\| \left(1-\chi_{1}\left(\h^{-2\delta}\mathcal{H}_{\h,\A}\right)\right)(\chi_{0}(q)\psi)\|^2\leq q_{\h\,A}\left( \left(1-\chi_{1}\left(\h^{-2\delta}\mathcal{H}_{\h,\A}\right)\right)\mathcal{H}_{\h,\A} (\chi_{0}(q)\psi)\right)\\
  \leq C_{1}\h\|
  \left(1-\chi_{1}\left(\h^{-2\delta}\mathcal{H}_{\h,\A}\right)\right)(\chi_{0}(q)\psi)\|^2+O(\h^{\infty})\|\psi\|^2.
\end{multline*}
Since $\delta\in\left(0,\frac{1}{2}\right)$, we deduce
\eqref{chi-chi0}.
\end{proof}

\subsection{Microlocalization of the eigenfunctions of
  $\mathcal{N}_{\h}$}
\label{sec:micr-eigenf-N}
The next two propositions state the microlocalization of the eigenfunctions of the normal form $\mathcal{N}_{\h}$.
\begin{prop}\label{loc-z2}
  Let us consider the pseudo-differential operator:
$$\mathcal{N}_{\h}=\mathcal{H}_{\h}^0+\Op_{\h}^w f^\star(\h,\mathcal{I}_{\h},\hat z_{2}).$$
We assume the confinement assumption \eqref{conf}. We can consider $\tilde M_{0}>0$ such that $B\circ\varphi^{-1}(\hat z_{2})\geq \tilde C_{1}$ for $|\hat
z_{2}|\geq \tilde M_{0}$. Let us consider $C_{1}<\tilde C_{1}$ and an
eigenpair $(\la,\psi)$ of $\mathcal{N}_{\h}$ such that $\la\leq
C_{1}\h$. Then, for all $\eps_{0}>0$ and for all smooth cutoff function $\chi$ supported in
$|\hat z_{2}|\geq \tilde M_{0}+\eps_{0}$, we have:
$$\Op_{\h}^w\left(\chi(\hat z_{2})\right)\psi=O(\h^{\infty}).$$
\end{prop}
\begin{proof}
  We notice that:
$$\mathcal{N}_{\h}\Op_{\h}^w\left(\chi(\hat z_{2})\right)\psi=\la\Op_{\h}^w\left(\chi(\hat z_{2})\right)\psi+\h\mathcal{R}_{\h}\psi,$$
where the symbol of $\mathcal{R}_{\h}$ is supported in compact
slightly smaller than the support of $\chi$. We may consider a cutoff
function $\underline{\chi}$ which is $1$ on a small neighborhood of
this support. We get:
\begin{multline*}
  \langle\mathcal{N}_{\h}\Op_{\h}^w\left(\chi(\hat z_{2})\right)\psi,\Op_{\h}^w\left(\chi(\hat z_{2})\right)\psi\rangle\leq \la\|\Op_{\h}^w\left(\chi(\hat z_{2})\right)\psi\|^2+C\h\|\Op_{\h}^w\left(\underline{\chi}(\hat z_{2})\right)\psi\|\|\Op_{\h}^w\left(\chi(\hat z_{2})\right)\psi\|\\
\end{multline*}
Thanks to the G\aa rding inequality, we have:
\begin{align*}
  \langle \mathcal{H}_{\h}^0 \Op_{\h}^w\left(\chi(\hat
    z_{2})\right)\psi, \Op_{\h}^w\left(\chi(\hat
    z_{2})\right)\psi\rangle\geq & (\tilde C_{1}-C\h)\|\Op_{\h}^w\left(\chi(\hat z_{2})\right)\mathcal{I}_{\h}^{1/2}\psi\|^2\\
  \geq & (\tilde C_{1}-C\h)\h\|\Op_{\h}^w\left(\chi(\hat z_{2})\right)\psi\|^2.
\end{align*}
We can consider $\Op_{\h}^w f^\star(\h, \mathcal{I}_{\h},\hat z_{2})$
as a perturbation of $\mathcal{H}_{\h}^0$  (see~\eqref{equ:relative}). Since $C_{1}<\tilde C_{1}$ we infer that:
$$\|\Op_{\h}^w\left(\chi(\hat z_{2})\right)\psi\|\leq C\h \|\Op_{\h}^w\left(\underline{\chi}(\hat z_{2})\right)\psi\|.$$
Then a standard iteration argument provides
$\Op_{\h}^w\left(\chi(\hat z_{2})\right)\psi=O(\h^{\infty})$.
\end{proof}

\begin{prop}\label{loc-z1}
  Keeping the assumptions and the notation of
  Proposition~\ref{loc-z2}, we consider
  $\delta\in\left(0,\frac{1}{2}\right)$ and an eigenpair $(\la,\psi)$
  of $\mathcal{N}_{\h}$ with $\la\leq C_{1}\h$. Then, we have:
$$\psi=\chi_{1}\left(\h^{-2\delta}\mathcal{I}_{\h}\right)\Op_{\h}^w\left(\chi_{0}(\hat z_{2})\right)\psi+O(\h^{\infty}),$$
for all smooth cutoff function $\chi_{1}$ supported in a
neighborhood of zero and all smooth cutoff function $\chi_{0}$ being $1$ near zero and supported in the ball of center $0$ and radius $\tilde M_{0}+\eps_{0}$.
\end{prop}
\begin{proof}
  The proof follows the same lines as for Proposition \ref{loc-z2} and
  Proposition \ref{micro-loc-L}.
\end{proof}

\subsubsection{Proof of Theorem \ref{spectrum}}

As we proved in the last section, each eigenfunction of
$\mathcal{H}_{\h,\A}$ or $\mathcal{N}_{\h}$ is
microlocalized. Nevertheless we do not know yet if all the functions
in the range of the spectral projection below $C_{1}h$ are
microlocalized. This depends on the rank of the spectral
projection. The next two lemmas imply that this rank does not increase
more than polynomially in $\h^{-1}$ (so that the functions lying in
the range of the spectral projection are microlocalized).  We will
denote by $N(\mathcal{M},\la)$ the number of eigenvalues of
$\mathcal{M}$ less than or equal to $\la$.
\begin{lem}\label{Nh}
There exists $C>0$ such that for all $\h>0$, we have:
$$N(\mathcal{H}_{\h,\A},C_{1}\h)\leq C\h^{-1}.$$
\end{lem}
\begin{proof}
We notice that:
$$N(\mathcal{H}_{\h,\A},C_{1}\h)=N(\mathcal{H}_{1,\h^{-1}A},C_{1}\h^{-1})$$
and that, for all $\eps\in\left(0,1\right)$:
$$q_{1,\h^{-1}\A}(\psi)\geq (1-\eps)q_{1,\h^{-1}\A}(\psi)+\eps\frac{B(x)}{\h}\|\psi\|^2$$
so that we infer:
$$N(\mathcal{H}_{\h,\A},C_{1}\h)\leq N(\mathcal{H}_{1,\h^{-1}\A}+\eps(1-\eps)^{-1}\h^{-1}B(x),(1-\eps)^{-1}C_{1}\h^{-1}).$$
Then, the diamagnetic inequality \footnote{See \cite[Theorem 1.13]{CFKS87} and the link with the control of the resolvent kernel in \cite{Kato72, Simon79}.} jointly with a Lieb-Thirring estimate (see the original paper \cite{ LT76})
provides for all $\gamma>0$ the existence of $L_{\gamma,2}>0$ such
that for all $\h>0$ and $\lambda>0$:
$$\sum_{j=1}^{N(\mathcal{H}_{1,\h^{-1}\A}+\eps(1-\eps)^{-1}\h^{-1}B(x),\lambda)}  \left|\tilde\lambda_{j}(\h)-\lambda\right|^\gamma\leq L_{\gamma,2}\int_{\R^2} (\eps(1-\eps)^{-1}\h^{-1}B(x)-\lambda)_{-}^{1+\gamma}\dx x.$$
We apply this inequality with
$\lambda=(1+\eta)(1-\eps)^{-1}C_{1}\h^{-1}$, for some $\eta>0$. This implies that:
$$\sum_{j=1}^{N_{\epsilon,\h,\eta}}
\left|\tilde\lambda_{j}(\h)-\lambda\right|^\gamma\leq
L_{\gamma,2}\int_{B(x)\leq(1+\eta)C_{1}/\epsilon}
(\lambda-\eps(1-\eps)^{-1}\h^{-1}B(x))^{1+\gamma}\dx x$$
with $N_{\epsilon,\h,\eta}:= N(\mathcal{H}_{1,\h^{-1}\A}+\eps(1-\eps)^{-1}\h^{-1}B(x),(1-\eps)^{-1}C_{1}\h^{-1})$,
so that:
\begin{multline*}
  (\eta(1-\eps)^{-1}C_{1}\h^{-1})^\gamma N_{\epsilon,\h,\eta}
  \leq L_{\gamma,2} (\h(1-\eps))^{-1-\gamma}\int_{B(x)\leq
    \frac{(1+\eta)C_{1}}{\eps}} ((1+\eta)C_{1}-\eps B(x))^{1+\gamma}\dx x.
\end{multline*}
For $\eta$ small enough and $\eps$ is close to 1, we have $
(1+\eta)\eps^{-1}C_{1}< \tilde C_{1}$ so that the integral is finite,
which gives the required estimate.
\end{proof}
\begin{lem}\label{H0}
There exists $C>0$ and $\h_{0}>0$ such that for all $\h\in(0,\h_{0})$, we have:
$$N(\mathcal{N}_\h,C_{1}\h)\leq C\h^{-1}.$$
\end{lem}
\begin{proof}
  Let $\epsilon\in(0,1)$. By point~\eqref{item:Q} of
  Theorem~\ref{main-theo}, it is enough to prove that
  $N(\mathcal{H}^0_\h,\frac{C_1\h}{1-\epsilon}) \leq Ch^{-1}$. The
  eigenvalues and eigenfunctions of $\mathcal{H}_\h^0$ can be found by
  separation of variables: $\mathcal{H}_\h^0 = \mathcal{I}_\h \otimes
  \Op_\h^w(B\circ\phy^{-1})$, where $\mathcal{I}_\h$ acts on
  $L^2(\R_{x_1})$ and $\hat B_\h:=\Op_\h^w(B\circ\phy^{-1})$ acts on
  $L^2(\R_{x_2})$. Thus,
\[
N(\mathcal{H}_\h^0, \h C_{1,\epsilon}) = \# \{(n,m)\in (\N^*)^2; \quad
(2n-1)\h \gamma_m(\h) \leq  \h C_{1,\epsilon}\},
\]
where $C_{1,\epsilon}:=\frac{C_1}{1-\epsilon}$, and
$\gamma_1(\h)\leq \gamma_2(\h) \leq \cdots$ are the eigenvalues of
$\hat B_\h$. A simple estimate gives
\[
N(\mathcal{H}_\h^0, C_{1,\epsilon}) \leq
\left(1+\left\lfloor\dfrac12+\frac{C_{1,\epsilon}}{2\gamma_1(\h)}\right\rfloor\right) \cdot \#\{m\in\N^*;
\quad \gamma_m(\h)\leq C_{1,\epsilon}\}.
\]
If $\epsilon$ is small enough, $C_{1,\epsilon}< \tilde{C_1}$, and then
Weyl asymptotics (see for instance \cite[Chapter 9]{DiSj99}) for $\hat{B}_\h$ gives
\[
N(\hat{B}_\h,C_{1,\epsilon}) \sim \frac{1}{2\pi\h}\textup{vol}\{B\circ\phy^{-1}\leq C_{1,\epsilon}\},
\]
and G\aa{}rding's inequality implies $\displaystyle{\gamma_1(\h)\geq \min_{q\in\R^2} B - O(\h)}$, which
finishes the proof.
\end{proof}

\begin{rem}
  With additional hypotheses on the magnetic field, it has been proved
  that the ${O}(\h^{-1})$ estimate is in fact optimal: see for
  instance~\cite{CdV86} and~\cite[Remark 1]{Truc97}. Actually, it
  would likely follow from Theorem~\ref{spectrum} and
  Theorem~\ref{main-theo} that these Weyl asymptotics hold in general
  under the confinement assumption.
\end{rem}
Let us now consider $\lambda_{1}(\h),\cdots, \lambda_{N(\mathcal{H}_{\h,\A},C_{1}\h)}(\h)$ the eigenvalues of $\mathcal{H}_{\h,\A}$ below $C_{1}\h$. We can consider corresponding normalized eigenfunctions $\psi_{j}$ such that : $\langle\psi_{j},\psi_{k}\rangle=\delta_{kj}$. We introduce the $N$-dimensional space:
$$V=\chi_{1}\left(\h^{-2\delta}\mathcal{H}_{\h,\A}\right)\chi_{0}(q)\underset{1\leq j\leq N}{\mathsf{span}}\psi_{j}.$$
Let us bound from above the quadratic form of $\mathcal{N}_{\h}$ denoted by $\mathcal{Q}_{\h}$. For $\psi\in\underset{1\leq j\leq N}{\mathsf{span}}\psi_{j}$, we let: 
$$\tilde\psi=\chi_{1}\left(\h^{-2\delta}\mathcal{H}_{\h,\A}\right)\chi_{0}(q)\psi$$ 
and we can write:
$$\mathcal{Q}_{\h}(U_{\h}^*\tilde \psi)=\langle U_{\h}\mathcal{N}_{\h}U_{\h}^*\tilde\psi,\tilde\psi\rangle=\langle U_{\h}U_{\h}^*\mathcal{H}_{\h,\A}U_{\h}U_{\h}^*\tilde\psi,\tilde\psi\rangle-\langle U_{\h}R_{\h}U_{\h}^*\tilde\psi,\tilde\psi\rangle.$$
Since $U_\h$ is microlocally unitary, the elementary properties of the
pseudo-differential calculus yield:
$$\langle U_{\h}\mathcal{N}_{\h}U_{\h}^*\tilde\psi,\tilde\psi\rangle=\langle\mathcal{H}_{\h,\A}\tilde\psi,\tilde\psi\rangle-\langle U_{\h}R_{\h}U_{\h}^*\tilde\psi,\tilde\psi\rangle+O(\h^\infty)\|\tilde\psi\|^2.$$
Then, thanks to Proposition \ref{micro-loc-L} and Lemma~\ref{Nh} we
may replace $\tilde\psi$ by $\psi$ up to a remainder of order $O(\h^\infty)\|{\tilde\psi}\|$:
$$\langle U_{\h}\mathcal{N}_{\h}U_{\h}^*\tilde\psi,\tilde\psi\rangle=\langle\mathcal{H}_{\h,\A}\psi,\psi\rangle-\langle U_{\h}R_{\h}U_{\h}^*\tilde\psi,\tilde\psi\rangle+O(\h^\infty)\|\tilde\psi\|^2$$
so that:
$$\langle U_{\h}\mathcal{N}_{\h}U_{\h}^*\tilde\psi,\tilde\psi\rangle\leq\lambda_{N}(\h)\|\psi\|^2+|\langle U_{\h}R_{\h}U_{\h}^*\tilde\psi,\tilde\psi\rangle|+O(\h^\infty)\|\tilde\psi\|^2$$
and:
$$\langle U_{\h}\mathcal{N}_{\h}U_{\h}^*\tilde\psi,\tilde\psi\rangle\leq\lambda_{N}(\h)\|U_{\h}^*\tilde\psi\|^2+|\langle U_{\h}R_{\h}U_{\h}^*\tilde\psi,\tilde\psi\rangle|+O(\h^\infty)\|U_{\h}^*\tilde\psi\|^2.$$
Let us now estimate the remainder term $U_{\h}R_{\h}U_{\h}^*\tilde\psi$. We have:
$$U_{\h}R_{\h}U_{\h}^*\tilde\psi=U_{\h}R_{\h}U_{\h}^*\underline{\chi_{1}}\left(\h^{-2\delta}\mathcal{H}_{\h,\A}\right)\tilde\psi=U_{\h}R_{\h}U_{\h}^*\underline{\chi_{1}}\left(\h^{-2\delta}\mathcal{H}_{\h,\A}\right)(U_{\h}^*)^{-1}U_{\h}^*\tilde\psi+O(\h^\infty)\|U_{\h}^*\tilde\psi\|,$$
where $\underline{\chi_{1}}$ has a support slightly bigger then the one of $\chi_{1}$. We notice that
\[
U_{\h}^*\underline{\chi_{1}}\left(\h^{-2\delta}\mathcal{H}_{\h,\A}\right)(U_{\h}^*)^{-1}=\underline{\chi_{1}}\left(\h^{-2\delta}U_{\h}^*\mathcal{H}_{\h,\A}(U_{\h}^*)^{-1}\right).
\] Let
us now apply \eqref{division} with
$D_\h=U_{\h}^*\mathcal{H}_{\h,\A}(U_{\h}^*)^{-1}$ to get:
$$R_{\h}=S_{\h,M}(U_{\h}^*\mathcal{H}_{\h,\A}(U_{\h}^*)^{-1})^M+K_{N}+O(\h^\infty)$$
so that:
$$\|U_{\h}R_{\h}U_{\h}^*\underline{\chi_{1}}\left(\h^{-2\delta}\mathcal{H}_{\h,\A}\right)\tilde\psi\|=O(\h^{2M\delta})\|U_{\h}^*\tilde\psi\|^2.$$
We infer that:
$$\mathcal{Q}_{\h}(U_{\h}^*\tilde \psi)\leq \lambda_{N}(\h)\|U_{\h}^*\tilde\psi\|^2+O(\h^{2M\delta})\|U_{\h}^*\tilde\psi\|^2.$$
From the min-max principle, it follows that:
$$\mu_{N}(\h)\leq \lambda_{N}(\h)+O(\h^{2M\delta}).$$
The converse inequality follows from a similar proof, using
Proposition~\ref{loc-z1} and Lemma~\ref{H0}. This ends the proof of Theorem~\ref{spectrum}.

\paragraph{Acknowledgments} The authors are deeply grateful to
Fr\'ed\'eric Faure for stimulating discussions.  They also thank
Bernard Helffer for useful suggestions on the proof of Lemma~\ref{Nh}.

\bibliographystyle{abbrv}

\bibliography{BIB}

\end{document}